\documentclass[12pt,letterpaper]{amsproc}
\pdfoutput=1

\usepackage{amsmath}

\usepackage{amsthm}

\def\essup{\operatorname{\text{ess\hskip 1mm sup}}}

\usepackage{amssymb,graphicx,mathrsfs,subfigure}

\newtheorem{lemma}{Lemma}
\def\essup{\operatorname{\text{ess\hskip 1mm inf}}}
\def\essuper{\operatorname{\text{ess\hskip 1mm sup}}}

\def\beq{\begin{equation}}
\def\eeq{\end{equation}}
\def\beqs{\begin{equation*}}
\def\eeqs{\end{equation*}}
\def\bbE{{\mathbb{E}}}

\def\cF{{\mathcal F}}
\def\Exp{\mathbb{E}}

\def\hx{\hat{x}}

\def\mL{{\mathcal{L}}}

\usepackage{booktabs}

\begin{document}

\title{Adaptive sampling for linear state estimation%
\footnote{Provisionally accepted for publication in the SIAM journal on control and optimization.}}

\author[M.~Rabi]{
  Maben Rabi}
\address[Maben~Rabi]{Department of Signals and Systems \\
  Chalmers University of Technology \\
  412 96 Gothenburg, Sweden.}

\author[G.~V.~Moustakides]{
  George~V.~Moustakides}
\address[George~V.~Moustakides]{Department of Electrical and
Computer Engineering\\ University of Patras\\ 26500 Rio, Greece.}

\author[J.~S.~Baras]{
  John~S.~Baras}
\address[John~S.~Baras]{Institute for Systems Research, and\\ the
Department of Electrical and Computer Engineering\\ University of
Maryland\\ College Park, MD 20742. USA.}

\date{\today}

\sloppy

\maketitle


\begin{abstract}
When a sensor has continuous measurements but sends limited messages
over a data network to a supervisor which estimates the state,
the available packet rate fixes the achievable quality of
state estimation.
When such rate limits turn stringent, the sensor's messaging
policy should be designed anew.
What are the good causal messaging policies~?
What should message packets contain~?
What is the lowest possible distortion in a causal estimate
at the supervisor~?
Is Delta sampling better than periodic sampling~?
We answer these questions under an idealized model
of the network and the assumption of perfect measurements at the sensor.
If the state process is a scalar, or a vector of low dimension,
then we can ignore sample quantization.
If in addition, we can ignore jitter in the transmission delays
over the network, then our search for efficient messaging policies
simplifies.
Firstly, each message packet should contain the value of the state at
that time. Thus a bound on the number of data packets becomes a bound
on the number of state samples.
Secondly, the remaining choice in messaging is entirely about the
times when samples are taken. 
For a scalar, linear diffusion process, we study the problem 
of choosing the causal sampling times that will give
the lowest  aggregate squared error
distortion. We stick to finite-horizons and impose a
hard upper bound~$N$ on the
number of allowed samples. We cast the design as a problem of
choosing an optimal sequence of stopping times. We reduce this
to a nested sequence of problems each asking for a single optimal stopping
time.
Under an unproven but natural assumption
about the least-square estimate at the supervisor,
 each of these single stopping problems 
are 
of standard form.
The optimal stopping times are random times when the
estimation error exceeds designed envelopes. For the case where
the state is a Brownian motion, we give analytically: the shape
of the optimal sampling envelopes, the shape of the envelopes
under optimal Delta sampling, and their performances.
Surprisingly, we find that Delta sampling performs badly.
Hence, when the rate constraint is a hard limit on the
number of samples over a finite horizon, we should
should not use Delta sampling.
\end{abstract}


\section{Introduction}
Networked control systems have some control loops
completed over data networks rather than
over dedicated analog wires or field buses.
In such systems, monitoring and
control tasks have to be performed under constraints on the amount
of information that can be communicated to the supervisor or control station.
These communication constraints limit the rate of packet transmissions
from sensor nodes to the supervisor node. Even at these limited
rates, the network communications can be less than
ideal:~the packets can be delayed and sometimes lost.
In the networked system, all of these communication
degradations  lower performance, and so these
effects must be accounted for during control design.
In this paper, we only account for the limit on the packet rates,
and completely ignore random delays and packet losses.

Sending data packets as per a periodic timetable works well when high
data rates are possible. Sending packets aperiodically
and at variable times becomes worthwhile only when the
packet rate limits get stringent like in an industrial  wireless network.
Conceptually, packet rate constraints can be of the following three
types: 1)~{\em{Average rate limit}}: This is a  `soft constraint'
and calls for an upper limit on the average number of
transmissions, 2)~{\em{Minimum waiting time between
transmissions}}: Under this type, there is a mandatory minimum
waiting time between two successive transmissions from the same
node, and 3)~{\em{Finite transmission budget}}: This is a `hard
constraint' and allows only up to a prescribed number of
transmissions from the same node over a given time window. In the
simplest version of the third type of constraint, we set the
constraint's window to be the problem's entire time horizon. In its other
variations, we can promote a steadier flow of samples and
avoid too many samples being taken in a short while. This we do
by cutting the problem's time horizon into many disjoint
segments and applying the finite transmission
budget constraint on every segment.

Notice that these different types of constraints can be mixed in
interesting ways. In this work, we will adopt the simple version
of the finite transmission budget, in which the budget window is the
same as the problem's time horizon.
We study a problem of state
estimation which is an important component of distributed control
and monitoring systems. In Specific, a scalar linear signal is
continuously and fully observed at a sensor which generates a
limited number of packets. A supervisor receives the sequence of
packets and, on its basis, maintains a causal estimate. 
Clearly, the fewer the packets allowed, the worse the distortion.
 The design
question is:  {\em{How should the packets be chosen by the sensor
to minimize the estimation distortion ?}} The answer to this
question employs the idea that: {\em{packets should be generated
only when they contain `sufficiently' new information.}} Adaptive
sampling schemes, or, event-triggered sampling schemes as they are
also called, exploit this idea and send samples at times
determined by the trajectory of the source signal being sampled.
In contrast, deterministic sampling chooses sample times according
to an extraneous clock.

But first, we will consider possible times when packets should be sent,
and allowable payloads they can carry. The times when packets are sent
must be causal times which, even if random, are stopping times w.r.t. the sensor's
observations process. Likewise, he payloads have to be measurable w.r.t.
the filtration generated by the observations process. 
The above restrictions are merely the demands of causality.
When we place some idealized assumptions about the network,
a simple and obvious choice of
payload emerges.

\subsection{Strong Markov property, idealized network, and choice of
payload}
For all the problems treated in this paper, we need the two clocks at the
sensor and the supervisor to agree, and of course report time correctly.
We will also assume that the state signal~$x_t$ is a strong Markov
process. This means that for {\it{any stopping time}}~$\tau$, and any
measurable subset~$A$ of the range of $x$, and for any time~$t\ge\tau$,
\begin{align*}
{\mathbb{P}}\left[ x_t \in A \left| {\mathcal{F}}_{\tau}^{x} \right. \right] 
& = 
{\mathbb{P}}\left[ x_t \in A \left| x_{\tau} \right. \right].
\end{align*}
Linear diffusions of course have the strong Markov property.

Let the sequence $\left\{ \tau_1, \tau_2, \ldots \right\}$ of
positive reals represent the sequence of times when the sensor
puts packets on the network. Let the sequence of binary words  $\left\{
\pi_1, \pi_2, \ldots \right\}$ denote the corresponding sequence of
payloads put out. Let the sequence $\left\{ \sigma_1, \sigma_2, \ldots
\right\}$ of non-negative reals denote the corresponding transmission
delays incurred by these packets. These delays can be random but still
must be independent of the signal process. The packet arrival
times at the supervisor, arranged in the order in which they were sent,
will be: $\left\{ \tau_1+\sigma_1, \ \tau_2+\sigma_2, \ \ldots \right\}$.
Let the positive integer~$l(t)$ denote the number of packets put out by
the sensor up to and including the time~$t$. We have:
\begin{align*}
l(t) & = \sup \left\{ i \left| \tau_i \le t \right. \right\}.
\end{align*}
A causal record of the sensor's communication activities is the transmit
process defined as follows:
\begin{align*}
{\mathbb{TX}}_t & = \begin{pmatrix}
                        \tau_{l(t)} \\
                        \pi_{l(t)}
                    \end{pmatrix}.
\end{align*}
When a packet arrives, the supervisor can see the source time
stamp~$\tau_j$, the payload~$\pi_j$, and of course the arrival
time~$\tau_j+\sigma_j$. Of course, we ignore quantization
noise in the time-stamps, with the result that
the supervisor can read both $\tau_j$ and $\tau_j+\sigma_j$ with
infinite precision. The causal record of what the supervisor
receives over the network is described by the random process defined as:
\begin{align*}
{\mathbb{RX}}_t & =
                 \sum_j{ \ \ 
{\mathbf{1}}_{ {\mbox{\scriptsize $ \begin{Bmatrix}
                      \tau_j + \sigma_j \le t, \ {\text{and}} \\
                       t < \tau_{j+1} + \sigma_{j+1}          
              \end{Bmatrix} $  }}
             } 
                   \begin{pmatrix}
                        \tau_{j} \\
                        \tau_{j} + \sigma_{j} \\
                        \pi_{j}
                    \end{pmatrix},
                      }
\end{align*}
where we have assumed that no two packets arrive at exactly the same
time, and that packets are received in exactly the order in which they
were sent. If we were to study the general case where packets can arrive
out of sequence, then the arguments below will have to be made more
delicate, but the conclusion below will still hold.

The supervisor's task is causal estimation. This fact restricts the way
in which $ {\mathbb{RX}}_t $  is used by the supervisor. Let the
count~$r(t)$ denote the number of packets received so far. Then, the
data in the hands of the supervisor at time~$t$ is the collection:
\begin{gather*}
r(t), \bigl\{  \left( \tau_{j}, \tau_{j} + \sigma_{j}, \pi_{j} \right)
               \left|  1 \le j \le r(t)   \right.
      \bigr\}
\end{gather*}
This is to be used to estimate the present and future values of the
state. 

Consider the question what of what the sensor should assign as payloads
to maximize information useful for signal extrapolation. In specific,
what should the latest payload~$\pi_{r(t)}$ be~?. If
the bit width of payloads is large enough to let us ignore quantization, then the
best choice of payload is the sample value at the time of generation,
namely~$x_{\tau_{r(t)}}$. Because of the strong Markov property, at
times $s\ge t$,
\begin{align*}
{\mathbb{P}} \Bigl[  x_s \in A
\left| \  {{x_{\tau_{r(t)}}}}, 
r(t), \bigl\{ \left( \tau_{j}, \tau_{j} + \sigma_{j}, \pi_{j} \right)
               \left|  1 \le j \le r(t)   \right.
      \bigr\}
\right.
\Bigr]
& =
{\mathbb{P}} \Bigl[  x_s \in A
\left| \ {{x_{\tau_{r(t)}}}} 
\right.
\Bigr],
\end{align*}
which means that if $\pi_{r(t)}$ carries $x_{\tau_{r(t)}}$ exactly, then
the future estimation errors are minimized. 
Therefore, the ideal choice of payload is the sample value. But what
about the practical non-zero quantization noise~? Again, the strong
Markov property implies that all the bits available should be used
to encode the current sample; the encoding scheme depends on the
distortion criterion for estimation.

If the packets do not arrive out of turn, the effect of packet delays
even when random is not qualitatively different from the ideal
case where all  transmission delays are zero. Nonzero delays can only
make the estimation performance worse. Hence, we will assume all packet
transit delays to be zero, and
$l(t) = r(t)$ always.

\subsection{Ignoring quantization noise}
 In most networks~\cite{wirelessHART,tilburyHandbook,kalleCAN}, the
packets are of uniform size and even when of variable size, have
at least a few bytes of header and trailer files. These segments
of the packet carry source and destination node addresses, a time
stamp at origin, some error control coding, some higher layer
(link and transport layers in the terminology of data networks)
data blocks and any other bits/bytes that are essential for the
functioning of the packet exchange scheme but which nevertheless
constitute what is clearly an overhead. The payload or actual
measurement information in the packet should then be at least of
the same size as these `bells and whistles'. It costs only
negligibly more in terms of network resources, of time, or of
energy to send a payload of five or ten bytes instead of two bits
or one byte when the overhead part of the packet is already five
bytes. This means that the samples being packetized can be
quantized with very fine detail, say with four bytes, a rate at
which the quantization noise can be ignored for low dimensional
variables. For Markov state processes, this means that all of
these bytes of payload can be used to specify the latest value of
the state. In other words, in essentially all packet-based
communication schemes, the right unit of communication cost is the
cost of transmitting a single packet.
The exact number of bits used to quantize
the sample is not important, as long as there are enough to make
quantization noise insignificant. There are of course special
situations where the quantization rate as well as the sample
generation rate matter. An example occurs in the internet
congestion control mechanism called Transmission Control
Protocol~(TCP)~\cite{jacobssonTCP} where a node estimates the
congestion state of a link through congestion bits added to
regular data packets. In this case, the real payload
in packets is irrelevant to the congestion state, and the
information on the congestion state is derived from the one or two
bits riding piggy-back on the data packets. The developments in
this paper do not apply to such problems where
quantization is important.

\subsection{Infinite Shannon capacity and well-posedness}
The continuous-time channel from the sensor to the supervisor is
idealized and noise-free. Even when a sequence of packets is delivered
with delays, the supervisor can recover perfectly the input
trajectory~${\{{\mathbb{TX}}\}}_0^T$ from the corresponding
trajectory of the output~${\{{\mathbb{RX}}\}}_0^{T}$.
The supervisor can read each time~$\tau_i$, and the sample
value~$x_{\tau_i}$ with infinite precision. Since the sensor has
and infinite range of choices for each~$\tau_i$, the channel has
infinite communication capacity in the sensor of Shannon.

But this does not render the sampling problem ill-posed. A packet
arriving at time contains the data:
$
l(  \tau_{i}  ), \tau_{i}, 
x_{\tau_{i}}.
$
Given  $ \left( \tau_{i}, x_{\tau_{i}} \right)$, the trajectory
of $x$ prior to $\tau_i$  is of no use for estimating
$\left\{x_s \left| s \ge  \tau_{i}  \right. \right\}$.
Therefore, it does not pay to choose $\tau_i$ cleverly so as to
convey extra news that the supervisor will use to extrapolate
future values of $x$. No such strategy can add to what the supervisor
already gets, namely the pair:~$\left(\tau_i, x_{\tau_i} \right)$.
There is nevertheless scope, and in fact a need for choosing
 $\tau_i$  cleverly so that the supervisor can use the silence
{\it{before}} $\tau_i$ to improve its state estimate {\it{before}}
$\tau_i$. But for the causal estimation problem the infinite Shannon
capacity does not sway the choice of sampling policies.

In summary, our assumptions so far are: (1)~the state is a strong Markov
process, (2)~the channel does not delay or lose packets, (3)~the
time-stamps~~$\tau_i$ are available with infinite precision
to the supervisor, and (4)~the sample value~$x_{\tau_i}$ is available
with infinite precision at the supervisor. Thus we have:
$\sigma_i = 0 \ \forall \ i, \ {\mathbb{RX}}_t = {\mathbb{TX}}_t \
\forall \ t, \ {\text{and}} \ r(t) = l(t) \ \forall \ t$.

%

\subsection{Relationship to previous works}
State estimation problems with communication rate constraints
arise in a wide variety of networked monitoring and control setups
such as sensor networks, wireless industrial
monitoring and control systems, rapid prototyping using a wireless
network, and multi-agent robotics. A recent overview of research
in Networked control systems including a variety of specific
applications is available from the special issue
\cite{antsaklisBaillieulSpecialIssue}.

Adaptive or event-triggered sampling
may also be used to model the functioning of various neural
circuits in the nervous systems of animals. After all, the neuron
is a threshold-triggered firing device whose operation is closely
related to  Delta-sampling. However, it is not presently clear if
the communication rate constraint adopted in this paper occurs in
biological Neural networks.

Adaptive sampling and adaptive timing of actuation have
been used in engineered systems for close to a hundred years. Thermostats
use on-off controllers which switch on or off at times when
the temperature crosses thresholds (subject to
some hysteresis). Delta-Sigma modulation~(Delta sampling) is an
adaptive sampling strategy 
used in signal processing and communication systems.
Nevertheless, theory has not kept up with practice.

{\em\bf{Timing of observations via Pull sampling and Push sampling:}}
The problem of choosing the time instants to sample sensor
measurements has received early attention in the literature.
Kushner~\cite{kushner-timing}, in 1964, studied the deterministic,
offline choice of
measurement times in a discrete-time, finite-horizon, LQG optimal
control problem.
He showed that the optimal deterministic
sampling schedule can be found by solving a nonlinear optimization
problem. Skafidas \& Nerode~\cite{skafidasNerode} allow the
online choice of  times for taking sensor measurements.
But these times are to be chosen
online by the controller rather than by the sensor.
Their conclusion is that for linear controlled systems,
the optimal choice of measurement
times can be made offline. Their offline
scheduling problem is the same as Kushner's deterministic one. 

A generalization of these problems of deterministic choice of
measurement times, is the {\em{Sensor scheduling problem,}} 
which has been studied for estimation, detection and control
tasks~\cite{meier67measurement,barasSensorScheduling,wuArapostathisTAC08}.
This problem asks for online schedules for gathering
measurements from different sensors available.
 However, the
information pattern for this problem is the same as in the works
of Kushner and of Skafidas \& Nerode.
Under this
information pattern, data flows from sensors to their
recipients is directed by the recipients.
Such sensor sampling is of the ``pull'' type. An
alternative is the ``push'' type of sampling where the
sensor itself
regulates the flow of its data. When only one sensor is available,
it has more information than the recipient and hence, its
decisions on when to communicate its measurements can
be better than decisions the supervisor can make. 
Adaptive sampling
is essentially the push kind of sampling.

{\em\bf{Lebesgue sampling and its generalizations:}}
The first analytic study of the communication
benefits of using event-triggered sampling was
presented in the 2002 paper of {\AA}str{\"o}m
\& Bernhardsson~\cite{astrom-bernhardsson-cdc}.
They 
treat a minimum variance control problem with the push type of
sampling. The control consists of impulses which reset the state
to the origin, but there is an upper limit on the rate at which
impulses can be applied. Under such a constraint, the design asks
for a schedule of the application times for the impulses. For
scalar Gaussian diffusions, they perform explicit calculations to
show that the application of impulses triggered by fixed levels is
more efficient than periodic application.

This has spurred further research in the area.
Our own work~\cite{rabi-baras-bahamas,rabi-moustakides-baras-med06,iwsm09,rabiJohanssonsCDC08}
 generalized the work of  {\AA}str{\"o}m \& Bernhardsson.
Their impulse control problem is equivalent to
the problem of sampling for causal estimation.
 In the setting of
discrete time, Imer \& Basar~\cite{imerBasarACC06two}
 study the problem of efficiently using a limited
number of discrete-time impulses.
For a finite-horizon LQG optimal
control problem, they use backward dynamic programming to show
that time-varying thresholds are optimal.
Henningsson~et.al.~\cite{cervinAutomatica} have
generalized to delays and transmission constraints
imposed by real data networks.

In the setting of discrete time, for infinite horizons,
Hajek~et.al.~\cite{hajekPagingITW,hajekPagingJournal}, have
treated essentially the same problem as us. They were the first to
point out that in the sequential decision problem, the two agents have
different information patterns. For a general Markov
state process, they describe as unknown
the jointly optimal choice of sampling policy and estimator.
For state processes which are symmetric random walks, they
show that the jointly optimal scheme uses adaptive sampling,
and that the corresponding estimator is the same `centered' estimator
one uses for deterministic sampling. We are unable to
prove a similar claim about the optimal estimator for our
continuous time problem. 

The study of optimal adaptive sampling timing leads to
{\em{Optimal stopping}} problems of Stochastic control, or
equivalently, to impulse control problems.
The information pattern of adaptive sampling complicates the
picture but methods of solving multiple stopping time problems of
standard form which are available in the
literature~\cite{carmonaDayanik} are indeed useful.

The work reported in this paper has been announced previously
in~\cite{rabiThesis, rabi-baras-bahamas,
rabi-moustakides-baras-cdc06}. 
In~\cite{rabi-moustakides-baras-med06},
the single sample case has been
dealt with in more detail than here.

\subsection{Contributions and outline of the paper}

For the finite horizon state estimation problem, we cast the
search for efficient sampling rules as sequential optimization
problem over a fixed number of causal sampling times. This we do
in section~\ref{problemFormulationSection} where, we formulate an
optimal multiple stopping problem with the aggregate quadratic
distortion over the finite time horizon as its cost function.
We restrict 
 the estimate at the supervisor to be that which would be optimal
under deterministic sampling. Following
Hajek~et.al.~\cite{hajekPagingITW,hajekPagingJournal},
we conjecture that when the state is a linear diffusion process,
 this estimate is indeed the least-square
estimate corresponding to the optimal sampling strategy.

In section~\ref{brownianMotionSection}, we take the simplified
optimal multiple stopping problem and solve it explicitly when the
state is the (controlled) Brownian motion process. The optimal
sampling policies are first hitting times of time-varying
envelopes by the estimation error signal. Our analytical solution
shows that for each of the sampling times, the triggering
envelopes are symmetric around zero and diminish monotonically in
a reverse-parabolic fashion as time nears the end of the horizon.
We also describe analytically, the performance of the class of
modified Delta sampling rules in which the threshold $\delta$ varies
with the number of remaining samples.
We point out a simple and recursive procedure for choosing the most
efficient of these Delta sampling policies.

For the Ornstein-Uhlenbeck process, in
section~\ref{ornsteinUhlenbeckSection}, we derive Dynamic
programming equations for the optimal sampling policy. We compute the
solution to these equations numerically. We are not able to say
whether an explicit analytic solution like for the Brownian motion
process is possible. We can say that the optimal sampling times are first hitting
times of time-varying envelopes by the estimation error signal.
These envelopes are symmetric around zero and diminish
monotonically as time nears the end of the horizon. Also derived
are the equations governing the performance of modified Delta
sampling rules and the most efficient among them is found through
a numerical search. Finally, in section~\ref{conclusionsSection},
we conclude and speculate on extensions to this work for other
estimation, control and detection problems.


\section{MMSE estimation and optimal sampling\label{problemFormulationSection}}

Under a deterministic time-table for the sampling instants, the
Minimum mean square error (MMSE) reconstruction for linear systems
is well-known and is straightforward to describe - it is the
Kalman filter with intermittent but perfect observations. The
error variance of the MMSE estimate obeys the standard Riccati
equation. In
Delta-sampling~\cite{norsworthy,gaborGyorfi,graySourceCodingbook},
also called Delta-modulation, a new sample is generated when the
source signal moves away from the previously generated sample
value by a distance $\delta$. By this rule, between successive
sample times, the source signal lies within a ball of radius
$\delta$ centered at the earlier sample. 
Such news of the state signal during an inter-sample interval, is
possible in adaptive sampling but never in deterministic sampling.
Because of this, the signal reconstruction under adaptive sampling
differs from that under deterministic sampling, and we will see below
what the difference is.

We will also set up an optimization problem where we seek an
adaptive sampling policy minimizing the distortion of the MMSE
estimator subject to a limit on the number of samples. 
    \begin{figure}[t]
     \centering
     \subfigure[Adaptive sampling for real-time estimation]%
     {
     \includegraphics[width=0.45\textwidth]{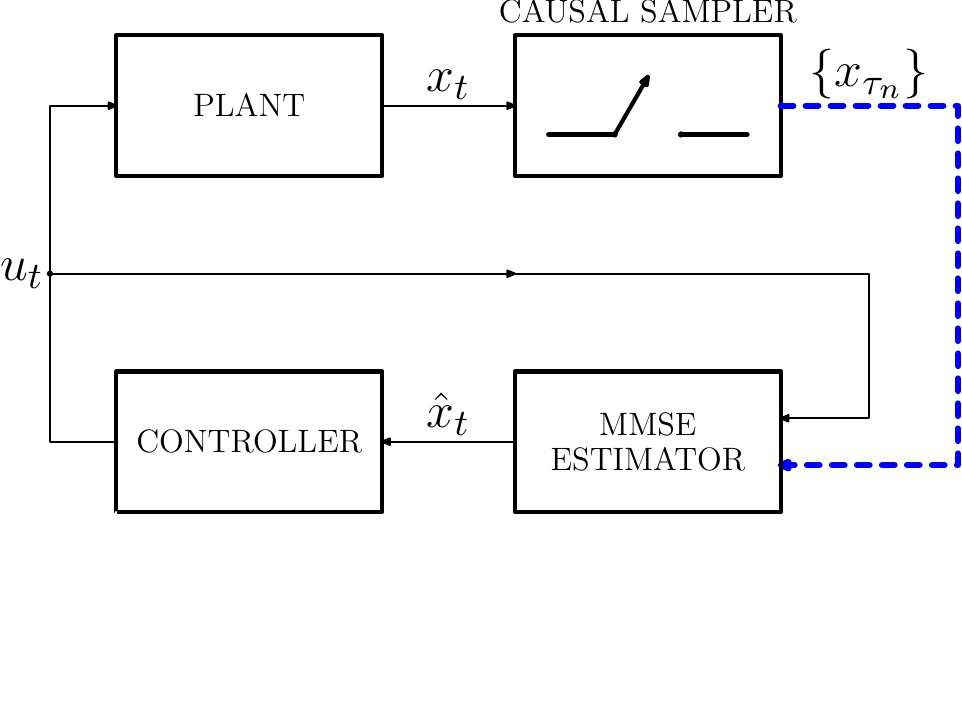}
     \label{bigPicture}
     }
     \subfigure[MMSE estimate under a time-varying threshold
     policy]%
     {
      \includegraphics[width=0.45\textwidth]{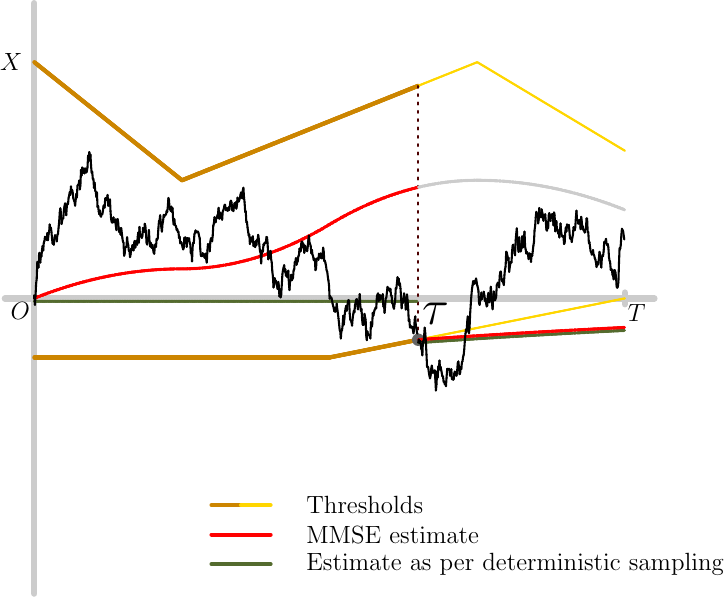}
      \label{generalETpicture}
     }
    \caption{Subfigure (a) depicts the setup for the MMSE estimation
    based on samples arriving at a limited rate. Subfigure (b)
    depicts the difference between estimators for Adaptive
    sampling and deterministic sampling.
     }
    \end{figure}
Consider a state process $x_t$ which is a (possibly controlled)
scalar linear diffusion. It evolves according to the SDE:
    \begin{align}
    dx_t & =  a x_t dt + b dB_t +  u_t dt, \ x_0 = x, \label{stateEquation}
    \end{align}
where, $B_t$  is a standard Brownian motion process. The control
process $u_t$ is RCLL and of course measurable with respect to the
$x$-process. In fact, the feedback form of $u_t$ is restricted to
depend on the sampled information only; we will describe this
subsequently. We assume that the drift coefficient $a$, the noise
coefficient $b\neq 0$, and, the initial value $x$ are known. Now,
we will dwell upon sampling and the estimation process.

The state is sampled at instants $\left\{ \tau_i \right\}_{i\ge
0}$ which are stopping times w.r.t. the $x$-process.
%
Recall that the process ${\mathbb{RX}}_t $ 
represents the data contained in the packets received at
the estimator:
\begin{align}
{\mathbb{RX}}_t & =
                 \sum_i{ \ \ 
{\mathbf{1}}_{ {\mbox{ \scriptsize $ \begin{Bmatrix}
                      \tau_i \le t, \ {\text{and}} \\
                       t < \tau_{i+1} 
              \end{Bmatrix} $ }}
             } 
                   \begin{pmatrix}
                        \tau_{i} \\
                        x_{\tau_{i}}
                    \end{pmatrix},
                      }
\label{receivedStream}
\end{align}
Notice that the binary-valued process ${\mathbf{1}}_{\left\{
\tau_i \le t\right\}}$ is measurable w.r.t.
${\mathcal{F}}_t^{{\mathbb{RX}}}$. The MMSE estimate ${\hat{x}}_t$ is
based on knowledge of the multiple sampling policy and all the
information contained in the output of the sensor and so it can
be written as:
    \begin{align*}
    {\hat{x}}_t & = \bbE \left[   x_t {\Big{|}} {\mathbb{RX}}_t \right].
    \end{align*}
The control signal $u_t$ is measurable w.r.t.
${\mathcal{F}}_t^{ {\mathbb{RX}}}$. Typically it is restricted to be of
the certainty-equivalence type as depicted in figure
\ref{bigPicture}. In that case $u_t$ is, in addition, measurable
w.r.t. ${\mathcal{F}}_t^{\hat{x}}$. The exact form of the feedback
control is not important for our work, but it is essential that
both the supervisor and sensor know the feedback control policy (and so has access
to the control waveform $u_0^t$). With this
knowledge, the control waveform is a known
additive component to the state evolution, and hence can be subtracted
out. Therefore, there is no loss of generality in considering only the
uncontrolled plant.
\subsection{MMSE estimation under deterministic sampling}
Consider now a deterministic scheme for choosing the sampling
times. Let the sequence of non-negative and increasing sampling
times be:
    \begin{align*}
    {\mathcal{D}} & = \left\{ d_0, d_1, \ldots  \right\}, \ \ d_0=0,
    \end{align*}
where, the times $d_i$ are all statistically independent of all
data about the state received after time zero. They can depend on
the initial value of the state $x_0$.

We will now describe the MMSE estimate and its variance. Consider
a time $t$ in the semi-open interval $\left[d_i, d_{i+1} \right)$.
We have:
    \begin{align*}
    {\hat{x}}_t & = \bbE \left[   x_t  {\Big{|}} {\mathbb{RX}}_t \right],\\
                & = \bbE \left[   x_t  {\Big{|}} d_i \le t < d_{i+1}, \left\{  \left( d_j , x_{d_j} \right) {\Big{|}} 0 \le j \le i \right\}
                \right],\\
                & = \bbE \left[   x_t  {\Big{|}} d_i \le t < d_{i+1},
                d_i, x_{d_i} \right], \\
                & = \bbE \left[   x_t  {\Big{|}} d_i, x_{d_i} \right],
    \end{align*}
where, we have used the Markov property of the state process and
the mutual independence, conditioned on $x_0$, of the state and
the sequence ${\mathcal{D}}$. Furthermore,
    \begin{align*}
    {\hat{x}}_t & = \bbE \left[  e^{a\left(t-d_i\right)} x_{d_i} + \int_{d_i}^t{e^{a\left(t-s\right)} b~dB_s}
                  + \int_{d_i}^t{e^{a\left(t-s\right)} u_s ds}{\Big{|}} d_i, x_{d_i}
                  \right],\\
                 & = e^{a\left(t-d_i\right)} x_{d_i} + \int_{d_i}^t{e^{a\left(t-s\right)} u_s
                  ds}.
    \end{align*}
Thus, under deterministic sampling, the MMSE estimate obeys a
linear ODE with jumps at the sampling times.
    \begin{align*}
    {\frac{d{{\hat{x}}_t}}{dt}} & = a {\hat{x}}_t + u_t, \
    {\text{for}} \ t \notin {\mathcal{D}}, \quad {\text{and}}, \quad
    {\hat{x}}_{t} \ = \ x_t, \ {\text{if}} \ t \in {\mathcal{D}}.\\
\intertext{The variance $p_t = \bbE \left[  \left( x_t -
{\hat{x}}_t\right)^2
 \right]$ is given by the well-known Riccati equation:}
    {\frac{d{p_t}}{dt}} & = 2ap_t + b^2, \ {\text{for}} \ t \notin {\mathcal{D}},
    \quad  {\text{and}}, \quad
    p_{t} \ = \ 0, {\text{if}} \ t \in {\mathcal{D}}.
    \end{align*}
The above description for the MMSE estimate and its variance is
valid even when the sampling times are random provided that these times
are independent of the state process except possibly the initial
condition. There too, the evolution
equations for the MMSE estimate and its error statistics remain
independent of the policy for choosing the sampling times; the
solution to these equations merely get reset with jumps at these
random times. On the other hand, adaptive sampling modifies the evolution
of the MMSE estimator as we will see next.

\subsection{The MMSE estimate under adaptive sampling}
Between sample times, an estimate
of the state using any distortion criterion is an estimate up to a
stopping time and this is the crucial difference from
deterministic sampling. Denote this estimate by~${\tilde{x}}_t$.
At time $t$ within the sampling interval
$\left[\tau_i, \tau_{i+1} \right)$, the MMSE estimate is given by:
    \begin{align*}
    {\tilde{x}}_t & = \bbE \left[   x_t  {\Big{|}} {\mathbb{RX}}_t  \right],\\
                & = \bbE \left[   x_t  {\Big{|}} \tau_i \le t < \tau_{i+1}, \left\{  \left( \tau_j , x_{\tau_j} \right) {\Big{|}} 0 \le j \le i \right\}
                \right],\\
                & = \bbE \left[   x_t  {\Big{|}} \tau_i \le t < \tau_{i+1},
                ~\tau_i,~x_{\tau_i} \right],  \ \ \ {\text{(Strong Markov
property)}}\\%
                & = x_{\tau_i} + \bbE \left[  x_{t} -  x_{\tau_i} {\Big{|}}   t - \tau_i < \tau_{i+1} - \tau_i,
                ~\tau_i,~x_{\tau_i} \right].
    \end{align*}
Similarly, its variance $p_t$ can be written as:
    \begin{align*}
    {p}_t & = \bbE \left[   {\left(x_t - {\hat{x}}_t\right)}^2
    {\Big{|}} \tau_i \le t < \tau_{i+1},
                ~\tau_i,~x_{\tau_i} \right].
    \end{align*}
Between samples, the MMSE estimate is an estimate up to a stopping
time because the difference of two stopping times is also a
stopping time. In general, it is different from the MMSE estimate
under deterministic sampling (see appendix \ref{appendix}). This
simply means that in addition to the information contained in
previous sample times and samples, there are extra partial
observations about the state. This information is the fact that
the next stopping time $\tau_{i+1}$ has not arrived. Thus, in
adaptive schemes, the evolution of the MMSE estimator is dependent
on the sampling policy. This potentially opens the possibility of
a {\em{timing channel}} \cite{bitsThroughQueues} for the MMSE
estimator.

Figure {\ref{generalETpicture}} describes a particular
(sub-optimal) scheme for picking a single sample. There are two
time-varying thresholds for the state signal, an upper one and a
lower one. The initial state is zero and within the two
thresholds. The earliest time within $[0,T]$, when the state exits
the zone between the thresholds is the sample time. The evolution
of MMSE estimator is dictated by the shape of the thresholds thus
utilizing information available via the timing channel.

\subsection{An optimal stopping problem}
We formalize a problem of sampling for optimal estimation over a
finite horizon.  We seek to minimize the distortion between
the state and the estimate ${\hat{x}}$. 
We conjecture that under optimal sampling,
\begin{align}
 {\tilde{x}}_t & = {\hat{x}}_t \ \ {\text{almost surely}}.
\label{conjecture}
\end{align}
If increments of the state process are not required to
have
symmetric PDFs, clearly
the conjecture is false~(see appendix~A).

 On the interval $[0,T]$, we seek for the state
process (\ref{stateEquation}), with the initial condition $x_0$,
an increasing and causal sequence of at most $N$ sampling times
$\left\{\tau_1, \ldots , \tau_N \right\}$, to minimize the
aggregate squared error distortion:
    \begin{align}
    J \left( T, N \right) & =  \bbE \left[  \int_0^T{ {\left( x_s -
    \hat{x}_s \right) }^2 ds }\right]. \label{distortion}
    \end{align}
Notice that the distortion measure does not depend on the initial
value of the state because, it operates on the error signal ($
 x_t - {\hat{x}}_t$) which is zero at time zero. Notice also that
the communication constraint is captured by an upper limit on the
number of samples. In this formulation, we do not get any reward
for using fewer samples than the budgeted limit.

The optimal sampling times can be chosen one at a time using a
nested sequence of solutions to optimal single stopping time
problems. This is because, for a sampling time $\tau_{i+1}$ which
succeeds the time $\tau_{i}$, using a knowledge of how to choose
the sequence $\left\{ \tau_{i+1}, \ldots , \tau_N \right\}$
optimally, we can obtain an optimal choice for
$\tau_{i}$ by solving over $[0,T]$ the optimal single stopping
time problem
    \begin{gather*}
    {\underset {\tau \ge 0} {\essup}   } \
     \bbE \left[
       \int_0^{\tau_i}{ {\left( x_s - \hat{x}_s \right)}^2 ds} +  J^{*}\left( T - \tau_{i}, N-i
       \right),
    \right]
    \end{gather*}
where, $ J^{*}\left( T - \tau_{i}, N-i \right)$ is the minimum
distortion obtained by choosing $N-i$ sample times
$\left\{\tau_{i+1},\ldots , \tau_{N} \right\}$ over the interval
$[\tau_{i},T]$. The best choice for the terminal sampling time
$\tau_N$ is based on a solving a single stopping problem. Hence we
can inductively find the best policies for all earlier sampling
times. Without loss of generality, we can examine the optimal
choice of the first sampling time $\tau_1$ and drop the subscript
$1$ in the rest of this section.

\subsubsection{The optimal stopping problem and the Snell envelope}
The sampling problem is to choose a single
${\mathcal{F}}^x_t$-stopping time $\tau$ on $[0,T]$ to minimize
    \begin{align*}
    F \left( T, 1 \right) & =  \bbE \left[  \int_0^{\tau} {( x_s -
    \hat{x}_s)}^2 ds   + J^{*} \left( T-\tau, N-1 \right) \right],
    \end{align*}
where,
    \begin{align*}
    J^{*} & = {\underset {\left\{ \tau_2, \ldots , \tau_N \right\}}
                         {\essuper}
              } \  %
    \bbE  \left[ J\left( T-\tau, N-1 \right) \right].
    \end{align*}
%
This is a stopping problem in standard form, and to solve it,
we can use the so-called {\em{Snell envelope}}
(see \cite{karatzas-shreve-finance} Appendix D and also
\cite{peskirShiryaev}):
    \begin{align*}
    S_t & = {\underset {\tau \ge t} {\essuper}   } \hskip 4mm \bbE
            \left[ \int_0^{\tau}{ {\left( x_s-\hat{x}_s \right)}^2  ds}
                    + J^{*} \left( T-\tau, N-1 \right)
                    {\Big{\vert}} {\mathcal{F}}^x_t \right] , \\
        & =    \int_0^t{ {\left( x_s-\hat{x}_s \right)}^2  ds}
                       + {\underset {\tau \ge t} {\essuper}} \hskip 4mm
                       \bbE \left[
                               \int_t^{\tau} {\left( x_s-\hat{x}_s \right)}^2  ds
                               +J^{*} \left( T-\tau, N-1 \right)
                               {\Big{\vert}} x_t
                            \right].
    \end{align*}
Then, the earliest time when the cost of stopping does not exceed
the Snell envelope is an optimal stopping time. Since the Snell
envelope depends only on the current value of the state and the
current time, we get a simple threshold solution for our problem.


\subsection{Extensions to nonlinear and partially observed systems}
When the plant is nonlinear, the MMSE estimate under deterministic sampling is
the mean of the Fokker-Planck equation, and is give by:
    \begin{align*}
    {\xi}_t & = \bbE \left[   x_t  {\Big{|}}
                \tau_{latest}, x_{\tau_{latest}} \right], \ \
                {\text{where}}, \tau_{latest} = \sup{\left\{ d_i \le t
                \right\}}.
    \end{align*}
Under adaptive sampling, this may not be the optimal choice of estimate.
To obtain a tractable optimization problem we can restrict the
kind of estimator waveforms allowed at the supervisor. Using
the Fokker-Planck mean above leads to a tractable stopping problem
as does using the zero-order hold waveform:
    \begin{align*}
    {\xi}_t & = x_{\tau_{latest}}.
    \end{align*}
However, even a slightly more general piece-wise constant
estimate:
    \begin{align*}
    {\xi}_t & = h \left( x_{\tau_{latest}}, \tau_{latest} \right) .
    \end{align*}
leads to a stopping problem of non-standard form because $\tau$ and $h$
have to be chosen in concert.

When the plant sensor has a noisy observations, or in the vector case case,
noisy partial observations, the sampling problem remains unsolved. The
important question now is: {\em{What signal at the sensor should
be sampled ? Should the raw sensor measurements be sampled and
transmitted, or, is it profitable to process them first ?}} We
propose a solution with a separation into local filtering and
sampling. Accordingly, the sensor should compute a continuous
filter for the state. The sufficient statistics for this filter
should take the role of the state variable. This means that the
sensor should transmit current samples of the sufficient
statistics, at sampling times that are stopping times w.r.t. the
sufficient statistics process.

In the case of a scalar linear system with observations corrupted
by white noise, the local Kalman filter at the sensor
${\hat{x}_t^{sensor}}$ plays the role of the state signal. The
Kalman filter obeys a linear evolution equation and so the optimal
sampling policies presented in this paper should be valid. In the
rest of the paper, we will investigate and solve the sampling
problem, first for the Brownian motion process and then for the
Ornstein-Uhlenbeck process.


\section{Sampling Brownian motion\label{brownianMotionSection}}
The sampling problem for Brownian motion with a control term added
to the drift  is no different from the problem without it. This is
because the control process ${\left\{u_t\right\}}_{t\ge 0}$ is
measurable w.r.t. ${\mathcal{F}}^{ {\mathbb{RX}} }_t$, whether it is a
deterministic feed-forward term, or a feedback based on the sampled
information. Thus for the estimation problem, we can safely set
the control term to be zero, to get:
    \begin{align*}
    dx_t & =   b dB_t , \ x_0 = x.
    \end{align*}
The diffusion coefficient $b$ can be assumed to be unity. If it is
not, we can simply scale time, and in the
${\tfrac{t}{b^2}}$-time, the process obeys a SDE driven by a
Brownian motion with a unit diffusion coefficient. We study the
sampling problem under the assumption that the initial state is
known to the MMSE estimator. Under deterministic sampling,
the MMSE estimate for this process is a
zero order hold extrapolation of received samples.

We study three important classes of sampling. The
optimal deterministic one is traditionally used, and it provides an
upper bound on the minimum distortion possible. The first adaptive
scheme we study is Delta sampling which is based on first hitting
times of symmetric levels by the error process. Finally, we
completely characterize the optimal stopping scheme by recursively
solving an optimal multiple stopping problem.

\subsection{Optimal deterministic sampling}
Given that the initial value of the error signal is zero, we will
show through induction that  uniform sampling on the interval
$[0,T]$ is the optimal choice of $N$ deterministic sample times. Call
the deterministic set of sample times as:
    \begin{align*}
    {\mathcal{D}} & = \left\{  d_1, d_2, \ldots, d_N  \
            \vert \ 0 \le d_i \le T, \ \ d_i{i-1} \le d_{i} \
            {\text{for}} \ i = 2,\ldots, N
               \right\}.
    \end{align*}
Then, the distortion takes the form:
    \begin{align*}
    J_{Deter} \left( T , N   \right)
     & = \int_{0}^{d_1}  { \Exp \left( x_s - {\hat{x}}_s\right) }^2 ds
       + \int_{d_1}^{d_2}  {\Exp \left( x_s - {\hat{x}}_s\right) }^2 ds
       + \ldots
       + \int_{d_N}^{T}  { \Exp \left( x_s - {\hat{x}}_s\right) }^2 ds.
    \end{align*}
Consider the situation of having to choose exactly one sample over
the interval $[T_1,T_2]$ with the supervisor knowing the state at
time $T_1$. The best choice of the sample time which minimizes the
cost $ J_{Deter} \left(T_2 - T_1 , ~1  \right)$ is the midpoint
${\tfrac{1}{2}}\left(T_2 + T_1\right)$ of the given interval. We
propose that the optimal choice of $N-1$ deterministic samples
over $[T_1,T_2]$ is the uniform one:
    \begin{align*}
    \left\{ d_1, d_2, \ldots d_{N-1}  \right\} & = \left\{ T_1 +
    i{\frac{T_2 -T_1}{N}} {\Big{\vert}}  i = 1,2, \ldots, N-1
    \right\},
    \end{align*}
which leads to a distortion equalling ${\tfrac{1}{2N}}{\left(
T_2-T_1 \right)}^2$. Let $ J^*_{Deter} \left(  T_2 - T_1 , ~N \right)$
be the minimum distortion over $\left[0,T_2-T_1\right]$ using
$N$ samples generated at deterministic  times.
Now, we carry out the induction step, and
obtain the minimum distortion over the set of $N$ sampling times
over $[T_1,T_2]$ to be:
    \begin{align*}
 & \min_{d_1}  \left\{  \int_{0}^{d_1}  { \left(x_s -
                   {\hat{x}}_s\right) }^2 ds +
                    \min_{ \{ d_2, d_2, \ldots d_{N}\} }
                    J_{Deter} \left( T_2 - T_1 -d_1 , ~N-1  \right)
                    \right\} \\ 
 &    =  \min_{d_1}  \left\{ {\frac{{d_1}^2}{2}} + {\frac{{\left(T_2 - T_1-d_1\right)}^2}{2N}}
                         \right\},\\
 &    =   \min_{d_1} \left\{ {\frac{ Nd_1^2
                                    + d_1^2 -2d_2\left(T_2 - T_1\right)
                 +{\left(T_2 - T_1\right)}^2   }{2N}}  \right\},\\
  &   =   \min_{d_1} \left\{ {\frac{
                          (N+1) {\left( d_1 - {\frac{1}{(N+1)}} \left(T_2 - T_1\right)
                     \right)}^2  +  {\left( 1 - {\frac{1}{(N+1)}} \right)}
                     {\left(T_2 - T_1\right)}^2
                          }
                                       {2N}} \right\},\\
  &   =  {\frac{1}{2(N+1)}}{\left(T_2 - T_1\right)}^2,
    \end{align*}
the minimum being achieved for $d_1 = {\tfrac{1}{N+1}}\left(T_2 -
T_1\right)$. This proves the assertion about the optimality of
uniform sampling among all deterministic schemes provided that the
supervisor knows the value of the state at the start time.

\subsection{Optimal Delta sampling\label{wiener-multiple-level}}
As described before, Delta sampling is a simple event-triggered
sampling scheme which generates a new sample whenever the input
signal differs from the last sample by a pre-specified threshold.
Delta sampling is really meant for infinite horizon problems as
it produces inter-sample intervals that are unbounded. Since we
have on our hands a finite horizon problem, we will use a time-out
at the end time of the problem's horizon. To make the most of this
class of rules, we allow the thresholds to vary with the past
history of sample times. Thus the supervisor can compute the sequence
of thresholds from the record of samples received previously.
Only the sensor can find the actual sample time since it also has full
access to the state signal.


More precisely, at any sampling time as well as at the start of
the horizon, the threshold for the next sampling time is chosen.
This choice is allowed to depend on the number of samples
remaining as well as the amount of time left till the end of the
horizon. We set $\tau_0 = 0$, and define thresholds and sampling
times times recursively. The threshold for the $i^{\text{th}}$
sampling time is allowed to depend on the values of the previous
sampling times, and so it is measurable w.r.t.
${\mathcal{F}}_t^{ {\mathbb{RX}}}$. Assume that we are given the policy
for choosing causally a sequence of non-negative thresholds
$\left\{\delta_1 , \delta_2, \ldots , \delta_N \right\}$. Then,
for $i=1,2,\ldots ,N$, we can characterize the sampling times
$\left\{\zeta_1 , \zeta_2, \ldots , \zeta_N \right\}$ as follows:
    \begin{align*}
    {\mathcal{F}}^{\delta_i} & \subset
                   {\mathcal{F}}^{\left(\tau_1, \ldots ,
                        \tau_{i-1}\right)}  \ {\text{if}} \ i \ > \ 1, \\
    \tau_{i,\delta_i} & = \inf \left\{ t \Bigl| t \ge
                                 \tau_{i-1,\delta_{i-1}},
                           \left\vert  x_t - x_{\tau_{i-1}} \right\vert \ge \delta_i
                           \Bigr. \right\}, \\
    \zeta_i & = \min \left\{ \tau_{i,\delta_i}, T \right\} 
  .
    \end{align*}
The first threshold $\delta_1$ depends only on the length of the 
horizon, namely~$T$.

The optimal thresholds can be chosen one at a time using solutions
to a nested sequence of optimization problems each with a single
threshold as its decision variable. This is because, knowing 
how to choose the sequence $\left\{ \zeta_{i+1}, \ldots , \zeta_{N}  \right\}$ optimally, we can
obtain an optimal choice for $\zeta_{i}$ by solving over
the optimization problem:
    \begin{gather*}
    {\inf_{\delta_i \ge 0}  } \
     \bbE \left[
       \int_0^{\zeta_i}{ {\left( x_s - \hat{x}_s \right)}^2 ds}
       +  J^{*}_{Thresh}\left( T - \zeta_{i},~N-i
                      \right)
    \right],
    \end{gather*}
where the cost function $J^{*}_{Thresh}\left( T - \zeta_{i},~N-i
\right)$ is the minimum aggregated distortion over $[T -
\zeta_{i},T]$ achievable using at most $N-i$ samples generated
using thresholds for the magnitude of the error signal. Hence, if
we know how to generate the last sample efficiently, we can
inductively figure out rules governing the best thresholds for
earlier sampling times.
    \begin{figure}[t]
    \begin{centering}
    \subfigure[Only one sample used
    \label{brownianThresholdsDiagram}]
    {%
    \includegraphics[width=0.45\textwidth]{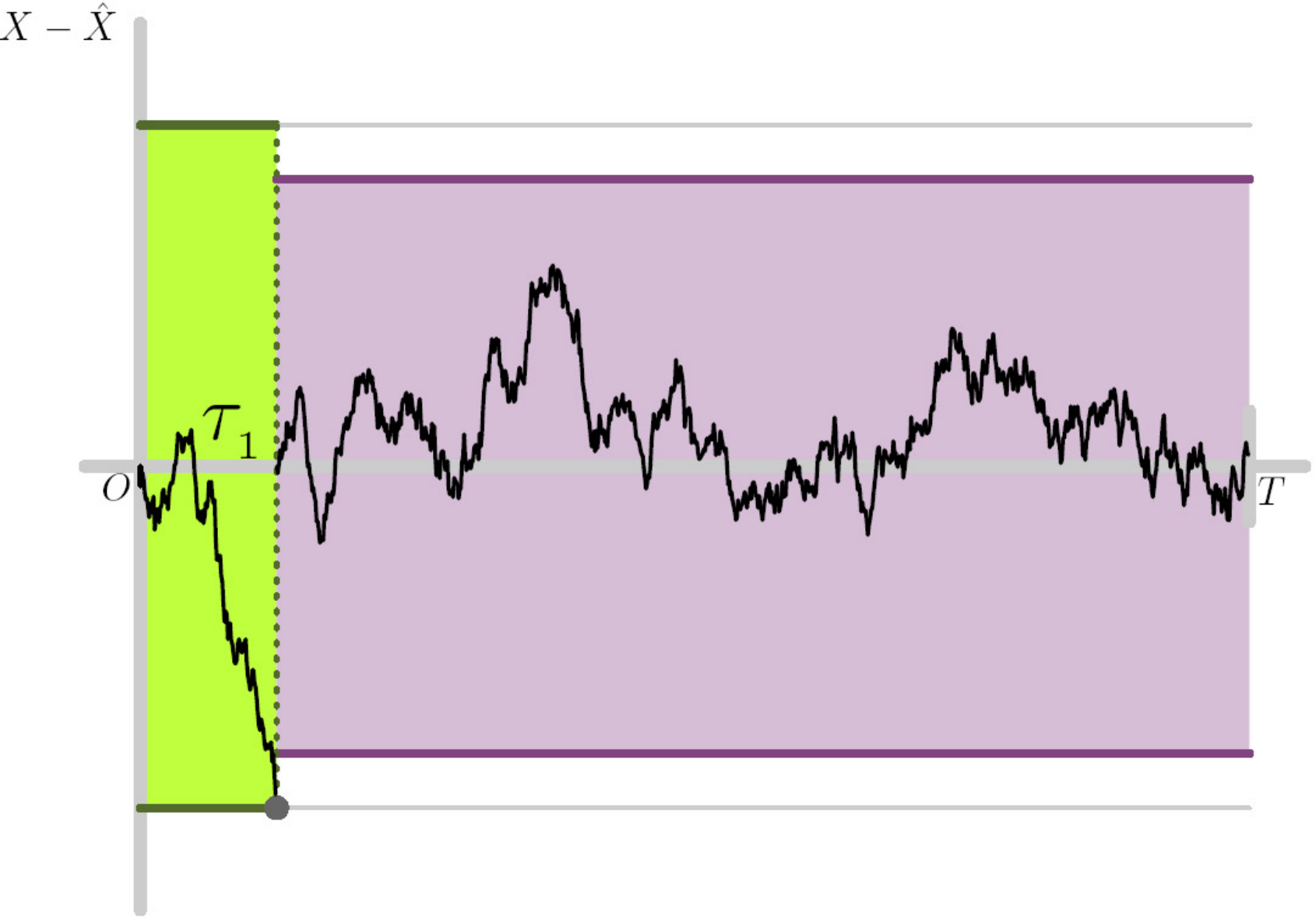}%
    }%
    \subfigure[Distortion as a function of threshold.
    \label{brownianThresholdsBehaviourDiagram}]
    {%
    \includegraphics[width=0.45\textwidth]{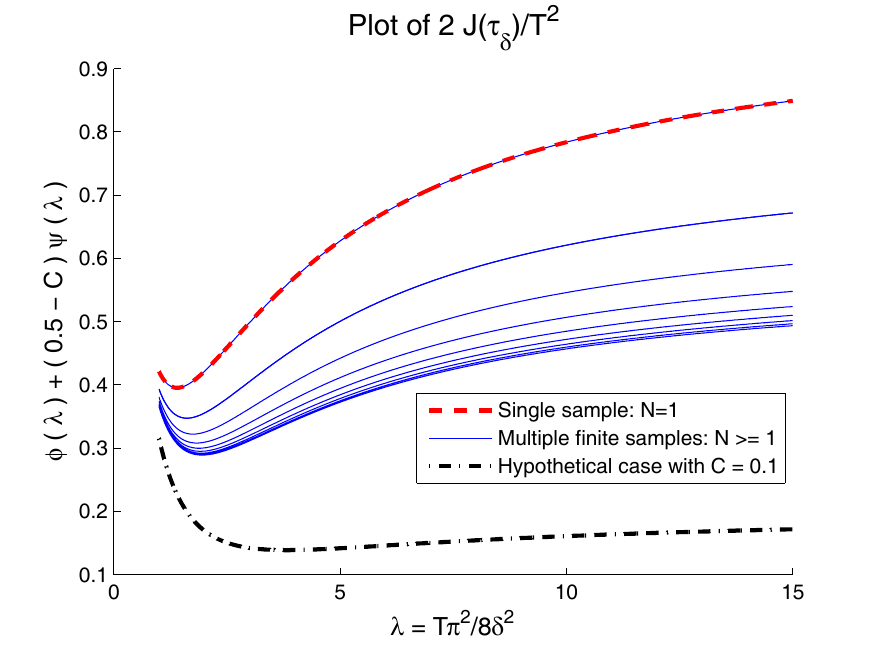}%
    }%
    \caption{In subfigure~(a), we depict Delta sampling. In
    subfigure~(b), We show the estimation distortion due to Delta sampling
    as a function of the threshold used. Notice that for a fixed $\delta$,
    the distortion
    decreases steadily as the maximum number of samples remaining~(N) grows.
    The distortion however never reaches zero. The minimum
    distortion reaches its lower limit of $0.287{\tfrac{T^2}{2}}$.}
    \end{centering}
    \end{figure}
\subsubsection{Optimal level for a single sample}
These computations are carried out in appendices B and C. In particular,
equation \ref{thresholdJforOneSample} gives the expression:
    \begin{align*}
    J_{Thresh} \left( T ,~1 \right) \left( \lambda \right) & =
        {\frac{T^2}{2}} \left\{ 1 +{\frac{\pi^4}{32\lambda^2}}%
          -{\frac{\pi^2}{4\lambda}}%
          -{\frac{\pi}{\lambda^2}}
              \sum_{k\ge 0}{(-1)^{k}} {\frac{
                  e^{-(2k+1)^2\lambda} }{(2k+1)^3}} \right\},%
    \end{align*}
where $\lambda = \frac{T\pi^2}{8\delta^2}$. parametrizing
in terms of  $\lambda$ reveals some structural
information about the solution. Firstly, note that the length of
the time horizon does not directly affect the optimum choice of
$\lambda$. The function
$J_{Thresh}\left( T ,~1 \right)$ has a shape that does not depend
on~$T$. It is merely scaled by the
factor ${\frac{T^2}{2}}$. The behaviour of the distortion as
$\lambda$ is varied can be seen in figure
\ref{brownianThresholdsBehaviourDiagram}. The minimal distortion
incurred turns out to be:
    $$
    c_1{\frac{T^2}{2}} =  0.3952  {\frac{T^2}{2}},
    $$
this being achieved by the choice: $\delta^* = 0.9391 \sqrt{T}$.
As compared to deterministic sampling, whose optimum performance
is $0.5\frac{T^2}{2}$, we realize that we have slightly more than
20\% improvement by using the optimum thresholding scheme.

{\em{How often does the Delta sampler actually generate a sample
?}} To determine that, we need to compute the probability that the
estimation error signal reaches the threshold before the end time
$T$. Using equation \ref{firingProbabilityExpression} provides the
answer: $98\%$. Note that this average sampling rate of the
optimal Delta sampler is independent of the length of the time
horizon.

We have provided a complete characterization of the performance of
the Delta sampler with one allowed sample. In what follows, we
will characterize the behaviour of multiple Delta sampling. In
doing so, we will hit upon the remarkable fact that if the sensor
is allowed two or more samples, it is actually more efficient to
sample deterministically.
\subsubsection{Multiple delta sampling}
Like in the single sample case, we will show that the expected
distortion over $[0,T]$ given at most $N$ samples is of the form
    $$
    c_N {\frac{T^2}{2}}.
    $$
Let $\tau_{\delta}$ be the level-crossing time as before. Then,
given a positive real number $\alpha$, consider the following
cost:
    \begin{align*}
    \Upsilon \left( T , \alpha, \delta \right)
       & {\stackrel{\Delta}{=}}
    \Exp \left[ \int_{0}^{ \tau_{\delta}\wedge T } x_s^2 ds
       + \alpha
       {\left[  { \left(  T - \tau_{\delta}
                             \right) }^{+}   \right]}^{2} \right].
    \end{align*}
Using the same technique as in the single sample case~(precisely,
the calculations between and including equations
\ref{thresholdRunningPlusTerminal},
\ref{singleThresholdFinalForm}), we get:
    \begin{align*}
    \Upsilon \left( T , \alpha, \delta \right)
       & = {\frac{T^2}{2}} -  \delta^2  \Exp  {\left[  {
    \left(  T - \tau_{\delta}
                                                  \right) }^{+}   \right]}
                                      - \left( {\frac{1}{2}} - \alpha
                                                  \right) \Exp  \left[
                             {\left[  { \left(  T - \tau_{\delta}
                             \right) }^{+}   \right]}^{2} \right].
    \end{align*}
Using calculations presented in appendix B, C we can
write~(equation~\ref{thresholdJforManySamples}):
    \begin{align*}
    \Upsilon \left( T , \alpha, \delta \right)
     &=  {\frac{T^2}{2}} \left\{ \phi(\lambda) +  \left[ {\frac{1}{2}} - \alpha
                 \right] \psi(\lambda) \right\}, 
    \end{align*}
where, $\lambda = \frac{T\pi^2}{8\delta^2}$, and we define the
functions $\phi, \psi$ as follows:
    \begin{align*}
     \phi\left(\lambda\right) & {\stackrel{\Delta}{=}} 1 +{\frac{\pi^4}{32\lambda^2}}%
              -{\frac{\pi^2}{4\lambda}}%
              -{\frac{\pi}{\lambda^2}}
                  \sum_{k\ge 0}{ \frac{ {(-1)^{k}}
                      e^{-(2k+1)^2\lambda} }{(2k+1)^3}},\\
\intertext{and,}
     \psi\left(\lambda\right) &{\stackrel{\Delta}{=}}
     -{\frac{5\pi^4}{96\lambda^2}}%
              -{\frac{\pi^2}{2\lambda}}%
              -2  +{\frac{16}{\pi\lambda^2}} \sum_{k\ge 0}{ \frac{ {(-1)^{k}}
                      e^{-(2k+1)^2\lambda} }{(2k+1)^5}}.
    \end{align*}%
The choice of $\lambda$ that minimizes the cost $\Upsilon$  can be
determined by performing a grid search for the minimum of the
scalar function $\phi(\lambda) +  \left[ {\frac{1}{2}} -
\alpha\right] \psi(\lambda)$. Since this sum is a fixed function,
we conclude that the  minimum cost is a {\it fixed} percentage of
${\frac{T^2}{2}}$ exactly as in the case of the single sample.
This property of this optimization problem is what enables us to
determine optimal multiple Delta samplers by induction.

Let us return to the distortion due to $N$ samples generated using
a Delta sampler, with the budget $N$ being at least $2$. If we
have determined the optimal Delta samplers for utilizing a budget
of $N-1$ or less, then the optimal distortion with a budget of $N$
samples takes the form:
    \begin{align*}
    J^*_{Thresh}\left( T ,~N \right)  & =
     {\inf_{\delta_{{}_N} \ge 0}  } \ \bbE \left[
       \int_0^{\zeta_N}{ {\left( x_s - \hat{x}_s \right)}^2 ds}
       +  J^{*}_{Thresh}\left( T - \zeta_{N},~N-1
                      \right)
    \right],\\
    & =
     {\inf_{\delta_{{}_N} \ge 0}  } \ \bbE \left[
       \int_0^{{\tau_{\delta_{{}_N}}\wedge T}}{ {\left( x_s - \hat{x}_s \right)}^2 ds}
       +  J^{*}_{Thresh}\Bigl( {\left(T - \tau_{\delta_{{}_N}}\right)}^+,~N-1
                      \Bigr)
    \right].\\
\intertext{Suppose the sensor is allowed to generate absolutely no
samples at all. Then the distortion at the supervisor will be:
${\frac{T^2}{2}}$. We know that the minimum distortion due to
using a single sample is a fixed fraction of ${\frac{T^2}{2}}$
namely: $c_1{\frac{T^2}{2}}$. We will now deduce by mathematical
induction that the minimum distortions possible with higher sample
budgets are also in the form of fractions of ${\frac{T^2}{2}}$.
Let the positive coefficient $c_k$ stand for the hypothetical
fraction whose product with ${\frac{T^2}{2}}$ is the minimum
distortion $J^*_{Thresh}\left( T ,~k \right)$. Continuing the
previous set of equations, we get:}
    J^*_{Thresh}\left( T ,~N \right) & =
    {\inf_{\delta_{{}_N} \ge 0}  } \  \bbE \left[
       \int_0^{{\tau_{\delta_{{}_N}}\wedge T}}
       { {\left( x_s - \hat{x}_s \right)}^2 ds}
       +  c_{N-1} {\left[ {\left(T -
       \tau_{\delta_{{}_N}}\right)}^+
                      \right]}^2\right],\\
    & = {\inf_{\delta_{{}_N} \ge 0}  } \  \Upsilon
    \Bigl( T , c_{N-1}, \delta_{{}_N} \Bigr), \\
    & = \ {\frac{T^2}{2}} \quad {\inf_{\lambda_N = \frac{T\pi^2}{8\delta_{{}_N}^2} } }
                \ \left\{ \phi(\lambda_N) +
                 \left[ {\frac{1}{2}} - c_{N-1}
                 \right] \psi(\lambda_N) \right\}.
    \end{align*}
Because of the scale-free nature of the functions $\phi , \psi$,
we have proved that the minimum distortion is indeed a fixed
fraction of ${\frac{T^2}{2}}$.
Figure~\ref{brownianThresholdsBehaviourDiagram} shows for different
values of~$N$, the graph of $J_{Thresh}\left(
T ,~N \right)$ as a function of $\lambda$. 
behaviour of the sum functions to be considered for efficient
multiple Delta sampling. The last equation gives us following
recursion for $k=1,2, \ldots , N$:
    \begin{gather}
    {\boxed{
    \begin{aligned}
    c_k & = \inf_{\lambda}  \left\{   \phi(\lambda) +
             (0.5-c_{k-1})\psi(\lambda)    \right\},
    &  \ & \ & \rho_k  & = {\frac{\pi}{2\sqrt{2\lambda_k^*}}}, \ {\text{and,}}\\
    \lambda_k^*  & = \arg\inf_{\lambda}  \left\{   \phi(\lambda) +
             (0.5-c_{k-1})\psi(\lambda)    \right\},
    &  \ & \ & \delta_k^* & = \rho_{N-k+1} \sqrt{T-\zeta_{k-1}}.
    \end{aligned}
    }}
    \label{optimalThresholdRecursion}
    \end{gather}
Now we determine the expected sampling rate of the optimal Delta
sampler. Let $\Xi_k$ be the random number of samples generated
before $T$ by the Delta sampler with a budget of $k$ samples. Then
almost surely, $\Xi_k$ equals the number of threshold crossings
generated by this sampler. Clearly, we have the bounds: $0\le
\Xi_k \le k$. Also, under optimal Delta sampling, the statistics
of the sampling rate do not depend on the length of the
time-interval $T$ as long as the latter is nonzero. This gives us
the recursion:
    \begin{align}
    \Exp \left[ \Xi_k   \right] & = \ 0 \cdot {\mathbb{P}}{\left[
    \tau_{{}_{\delta^*_k}}  \ge T \right]} \ + \ \left( 1 +
    {\mathbb{E}}{\left[ \Xi_{k-1} \right]} \right)\cdot
    {\mathbb{P}}{\left[ \tau_{{}_{\delta^*_k}} < T \right]}.%
    \label{samplingRateRecursion}
    \end{align}
We have derived the performance characteristics of optimal Delta
sampling with a fixed sample budget. The parameters describing the
performance are tabulated in \ref{tableDeltaSampling} for small
values of the budget.
    \begin{table}[h!]
      \centering
    \begin{tabular}{|l||l|l|l|l|l|} \hline
    $N$&1&2&3&4&5\\
    \hline\hline %
    $c_N$  &0.3953&0.3471&0.3219&0.3078&0.2995\\ \hline
    $\rho_N$ &0.9391&0.8743&0.8401&0.8208&0.8094\\ \hline
    $\Exp\left[\Xi_N\right]$    &0.9767&1.9306&2.8622&3.7541&4.4803\\
    \hline
    \end{tabular}
    \caption{Characteristics of Optimal Multiple Delta sampling for
    small values of the sample budget.}\label{tableDeltaSampling}
    \end{table}
    \begin{figure}[t]
    \begin{centering}
    \subfigure[Firing rate of Delta sampling]%
    {%
    \includegraphics[width=0.45\textwidth]{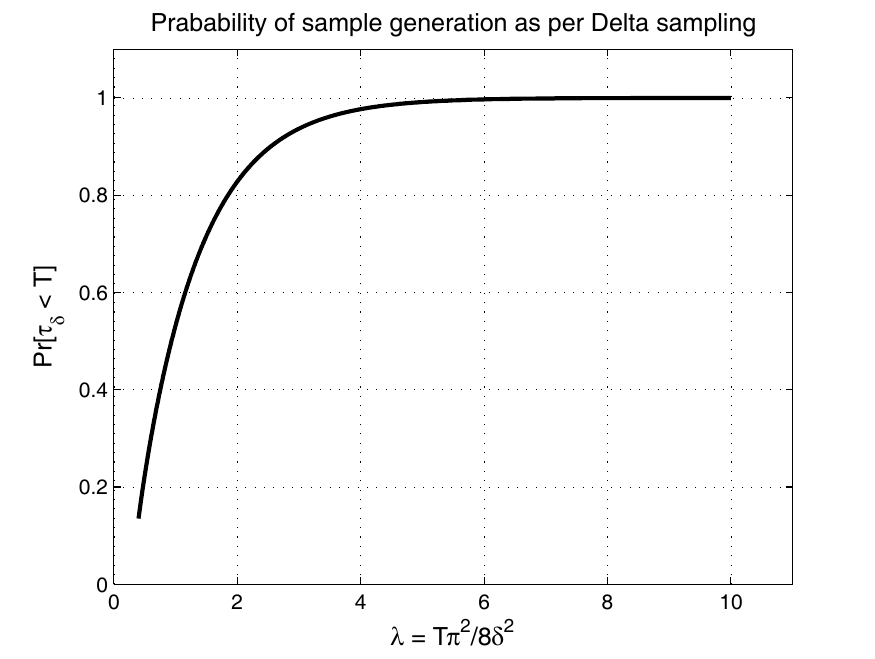}%
    }%
    \subfigure[Average rate of sample generation]%
    {%
    \includegraphics[width=0.45\textwidth]{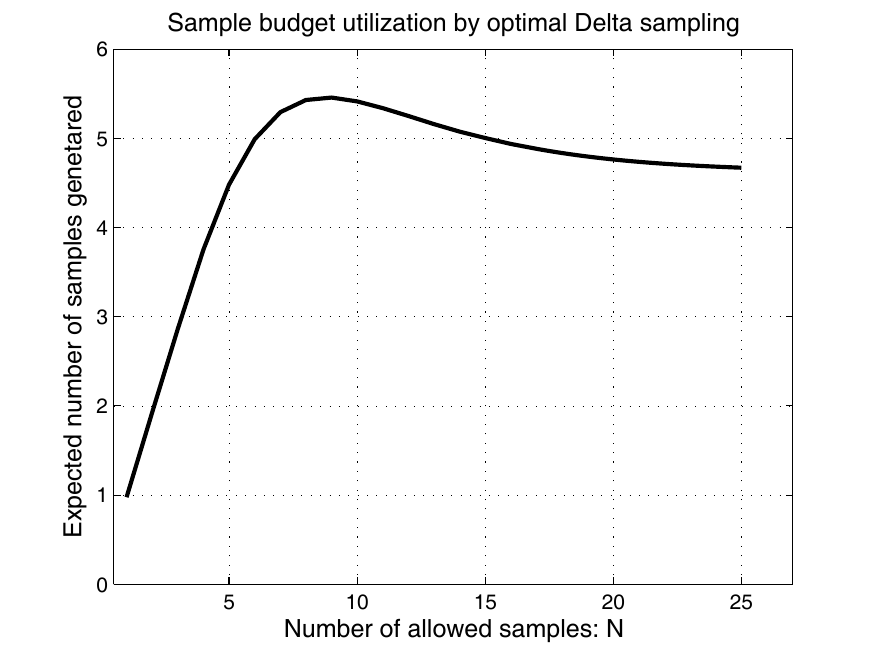}%
    }%
    \caption{Subfigure (a) shows the probability that a sample is generated
    ($\Xi$) as a function of the parameter $\lambda$, which is inversely related
    to the square of the threshold $\delta$.  Subfigure (b) demonstrates the fact
    that delta sampling is meant for repeated sampling over infinite horizons.
    The utilization of the sample budget by the optimal delta sampling scheme
    utilizes is not monotonic with the size of the budget and is actually
    somewhat counter-intuitive. Over any fixed finite horizon,  it uses fewer than
    six samples on average.}
    \label{brownianThresholdsFiringRates}
    \end{centering}
    \end{figure}
To understand the behaviour of optimal Delta sampling when the
sample budget is larger than five, look at subfigure (b) in
\ref{brownianThresholdsFiringRates} as well as figure
\ref{brownianComparisons}. The minimum distortion decreases with
increasing sample budgets but it does not decay to zero. It
stagnates at approximately $0.3{\frac{T^2}{2}}$ no matter how
large a budget is provided. The expected number of samples is not
monotonic in its dependence on the budget. It settles to a value
close to $4.5$. Clearly, Delta sampling which is optimal for the
infinite horizon version of the estimation problem is far from
optimal in the finite horizon version. In fact if the sample
budget is at least two, then, even deterministic sampling performs
better. There is hence a need for determining the optimal sampling
policy.

In optimal Delta sampling, the sensor chooses a sequence of
thresholds to be applied on the estimation error signal. The
choice of a particular threshold is made at the time of the
previous sample and is allowed to depend on the past history of
sample times. {\em{Suppose now that the sensor is allowed to
modify this choice causally and continuously at all time
instants}}. Then we get a more general class of sampling policies
with a family of continuously varying envelopes for the estimation
error signal. This class of policies happens to contain the optimal sampling
policy which achieves the minimum possible distortion. Next, we
will obtain the optimal family of envelopes by studying the
problem of minimum distortion as an optimal multiple stopping
problem.

\subsection{Optimal sampling}
Now, we go after the optimal multiple sampling policy. Consider
the non-decreasing sequence: $\left\{ \tau_1, \tau_2, \ldots ,
\tau_{{}_N} \right\}$ with each element lying within $[0,T]$. For
this to be a valid sequence of sampling times, its elements have
to be stopping times w.r.t. the $x$-process. We will look for the
best choice of these times through the optimization:
    \begin{align*}
    J^{*}\left( T ,~N\right)
     & =  {\underset {\left\{ \tau_1, \tau_2, \ldots , \tau_N \right\}}
                         {\essup}
              } \ %
    \bbE  \left[
        \int_0^{\tau_1} {x_s}^2 ds
      + \int_{\tau_1}^{\tau_2} {( x_s -\hat{x}_{\tau_{1}})}^2 ds
      + \cdots \right. \\
    & \hskip 25mm \left. + \int_{\tau_{{}_N}}^{T} {( x_s - \hat{x}_{\tau_{{}_N}})}^2 ds
         \right] .
    \end{align*}
The solution to this optimization parallels the developments for
Delta sampling. In particular, the minimum distortion obtained by
optimal sampling will turn out to be a fraction of
${\frac{T^2}{2}}$. We will recursively obtain optimal sampling
policies by utilizing the solution to the following optimal
(single) stopping problem concerning the objective function
$\chi$:
    \begin{align*}
    {\underset{\tau}{\essup}} \ \chi\left( T ,\beta, \tau
                         \right)
    & =  {\underset{\tau}{\essup}} \  %
    \bbE  \left[ \int_0^{\tau} {x_s}^2 ds
         + {\frac{\beta}{2}}{\left( T - \tau \right)}^2 \right],
    \end{align*}
where, $\tau$ is a stopping time w.r.t. the $x$-process that lies
in the interval $[0,T]$, and, $\beta$ is a positive real number.
We reduce this stopping problem into one having just a terminal
cost using the calculations between and including equations
\ref{thresholdRunningPlusTerminal},
\ref{singleThresholdFinalForm}:
    \begin{align*}
    \chi\left( T ,\beta, \tau
                         \right)
    & = {\frac{T^2 }{2}}
        - \Exp \left[
                   2x_{\tau}^2 \left(T-\tau\right) + \left(1-\beta\right)
                           {\left[ \left( T-\tau\right) \right]}^2
              \right],
    \end{align*}
which can be minimized by solving the following optimal stopping
problem:
    \begin{gather*}
       {\underset{\tau}{\essuper}}  \ \Exp \left[
                   2x_{\tau}^2 \left(T-\tau\right) + \left(1-\beta\right)
                           {\left[ \left( T-\tau\right) \right]}^2
              \right]
    \end{gather*}
This stopping problem can be solved explicitly by determining it
Snell envelope process. We look for a ${C}^2$ function $g
\left(x,t\right)$ which satisfies the Free boundary PDE system:
    \begin{align}
     {\frac{1}{2}}g_{xx} + g_t & = 0, \ {\text{and}} \quad\quad
     &%
     g \left(x,t\right) & \ge 2x^2 \left(T-t\right)
            + \left(1-\beta\right){\left(T-t\right)}^2.
     \label{freeBoundaryPDE}
    \end{align}
Given a solution $g$, consider the process:
    \begin{align*}
     S_t &{\stackrel{\Delta}{=}} g \left(x_t,t\right).
    \end{align*}
This is in fact the Snell envelope. To see that, fix a
deterministic time $t$ within $[0,T]$ and verify using It\^{o}'s
formula that:
    \begin{align*}
    \Exp\left[S_\tau(x_\tau)|x_t\right]-S_t=\Exp\left[\int_t^\tau
    dS_t|x_t\right]=0,
    \end{align*}
for any stopping time $\tau \in [t,T]$, and hence,
    \begin{align*}
    S_t  & = \Exp\left[S_\tau|x_t\right] \
    \ge \ \Exp\left[ x_\tau^2(1-\tau)|\tau \ge t, x_t\right].
    \end{align*}
The last equation confirms that $S_t$ is indeed the Snell
envelope. Consider the following solution to the Free-boundary PDE
system:
\begin{equation*}
g(x,t) = A \left\{
                       {\left( T-t \right) }^2 + 2 x^2
                      \left( T-t \right) + {\frac{x^4}{3}}
           \right\}.
\end{equation*}
where $A$ is a constant chosen such that $g(x,t) - 2x^2(T-t) -
(1-\beta)(T-t)^2$ becomes a perfect square. The only possible
value for $A$ then is:
\begin{equation*}
{\frac{ (5+\beta) - \sqrt{{(5+\beta)}^2 - 24} }{4}}.
\end{equation*}
Then an optimal stopping time is the earliest time the reward
obtained by stopping now equals or exceeds the Snell envelope:
\begin{equation*}
\begin{split}
\tau^* & = \inf_t \left\{ t: S_t \le 2x_t^2(T-t) +
                               (1-\beta)(T-t)^2 \right\}, \\
& =  \inf_t \left\{ t: x_t^2 \ge \sqrt{ {\frac{3(A-1+\beta)}{A}}
}(T-t) \right\},
\end{split}
\end{equation*}
and the corresponding minimum distortion becomes
    \begin{equation*}
    \chi^*\left( T ,\beta
                         \right)= (1-A) {\frac{T^2}{2}}.
    \end{equation*}
    \begin{figure}[t]
    \begin{centering}
    \includegraphics[width=0.5\textwidth]{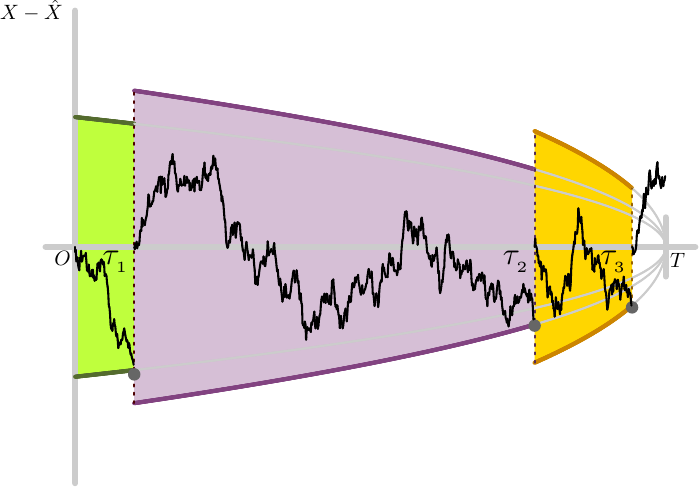}
    \caption{Parabolic envelopes for the estimation error signal under the
    optimal sampling scheme. 
    For each successive sample, the envelope becomes wider.}
    \label{brownianOptimalEnvelopesDiagram}
    \end{centering}
    \end{figure}
We now examine the problem of choosing optimally a single sample.
\subsubsection{Optimal choice of a single sample}
The minimum distortion due to using exactly one sample is:
    \begin{align*}
    J^{*}\left( T ,~1\right)
     & =   {\underset{\tau_1}{\essup}} \ %
    \bbE  \left[
        \int_0^{\tau_1} {x_s}^2 ds
      + \int_{\tau_1}^{T} {( x_s -\hat{x}_{\tau_{1}})}^2 ds
       \right], \\
    & =   {\underset{\tau_1}{\essup}} \ %
    \bbE  \left[
        \int_0^{\tau_1} {x_s}^2 ds
      + {\frac{1}{2}} {\left( T- {\tau_{1}} \right)}^2 ds
       \right],\\
    & =   {\underset{\tau_1}{\essup}} \ %
    \chi\left( T ,1, \tau_1
                         \right).
    \end{align*}
We have thus reduced the optimization problem to one whose
solution we already know. Hence, we have:
\begin{align*}
\tau_{1}^* & =   \inf_{t \ge 0}  \left\{ t :  { x_t }^2  \ge
\sqrt{3}\left(T-t\right) \right\}, \ {\text{and}}, \quad
&
J^{*}\left( T ,~1\right) & =
\left(\sqrt{3}-1\right){\frac{T^2}{2}}.
\end{align*}
\subsubsection{Optimal multiple sampling}
We obtain the family of policies for optimal multiple sampling by
mathematical induction. Suppose that the minimum distortions due
to using no more than $k-1$ samples over $[0,T]$ is given by the
sequence of values $\left\{\theta_{1}{\frac{T^2}{2}}, \ldots ,
\theta_{k-1}{\frac{T^2}{2}}\right\}$. Then consider the minimal
distortion due to using up to $k$ samples:
    \begin{align*}
    J^{*}\left( T ,~k\right)
     & =   {\underset{ \tau_1 }{\essup}} \ %
    \bbE  \left[
        \int_0^{\tau_1} {x_s}^2 ds
      +     J^{*}\left( T-\tau_1 ,~k-1\right)
       \right], \\
    & =   {\underset{\tau_1}{\essup}} \ %
    \bbE  \left[
        \int_0^{\tau_1} {x_s}^2 ds
      + {\frac{1-\theta_{k-1}}{2}} {\left( T- {\tau_{1}} \right)}^2 ds
       \right],\\
    & =   {\underset{\tau_1}{\essup}} \ %
    \chi\left( T ,\theta_{k-1}, \tau_1
                         \right).
    \end{align*}
This proves the hypothesis that the minimum distortions for
increasing values of the sample budget form a sequence with the
form: ${\left\{\theta_{k}{\frac{T^2}{2}}\right\}}_{k\ge 1}$. The
last equation also provides us with the recursion which is started
with $\theta_0 = 1$:
\begin{gather}
{\boxed{
\begin{aligned}
\theta_k & =  1 - {\frac{ (5+\theta_{k-1}) -
\sqrt{{(5+\theta_{k-1})}^2 - 24}}{4}} , \\
\gamma_k & =
\sqrt{{\frac{3(\theta_{k-1}-\theta_k)}{1-\theta_{k}}}}, \\
\tau_{k}^* & =   \inf_{t \ge \tau_k}  \left\{ t :  {\left( x_t -
                                    x_{\tau_k}  \right)}^2  \ge
                                    \gamma_{N-k+1} {T-t}
                                    \right\}.
\end{aligned}
}}
\end{gather}
\subsection{Comparisons}
    \begin{figure}[t]
    \begin{centering}
    \subfigure[Comparisons]%
    {%
    \includegraphics[width=0.45\textwidth]{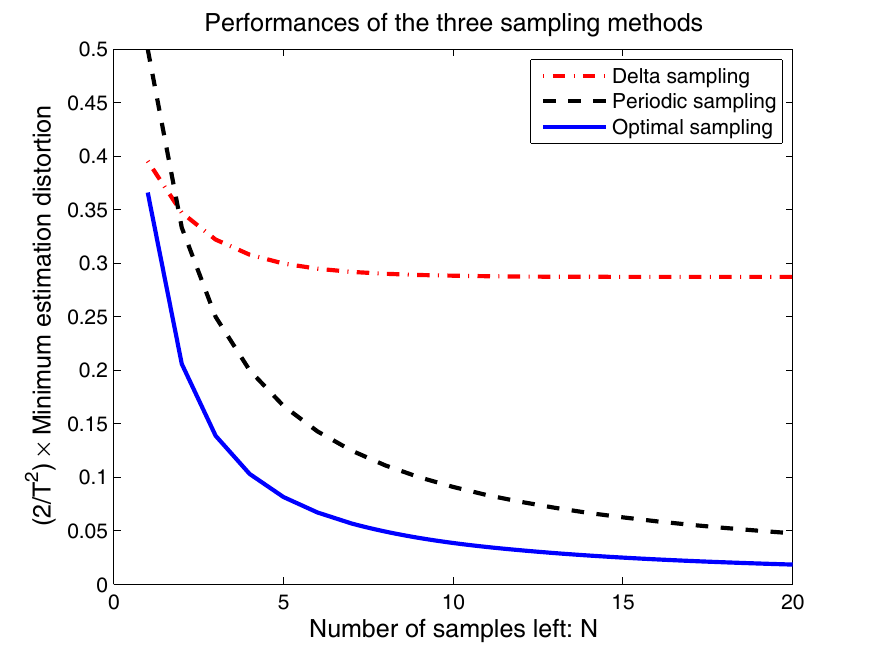}%
    }%
    \subfigure[Optimal vs Periodic]%
    {%
    \includegraphics[width=0.45\textwidth]{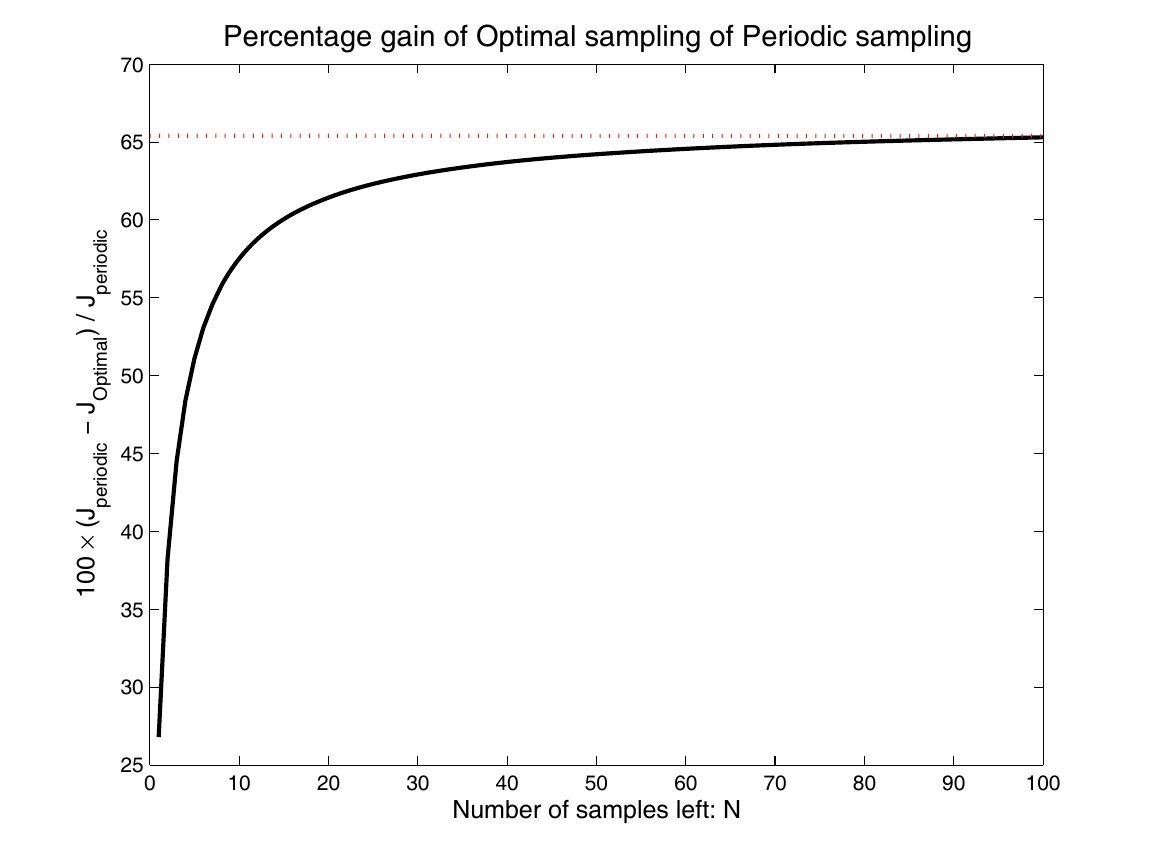}%
    }%
    \caption{The minimum distortions offered by the three sampling methods.
    As the number of allowed samples grows, the normalized reduction of the
    distortion of optimal sampling from periodic sampling keeps growing up
    to a limit of 67\%.}
    \label{brownianComparisons}
    \end{centering}
    \end{figure}
In figure~\ref{brownianComparisons} we have a comparison of the
estimation distortions incurred by the three sampling strategies.
The remarkable news is that
Delta sampling which is optimal for the infinite horizon version
of the estimation problem is easily beaten by the best
deterministic sampling policy. There is something intrinsic to
Delta sampling which makes it ill suited for finite horizon
problems with hard budget limits. This also means that it is not
safe to settle for `natural' event-triggered sampling policies
such as Delta sampling. Also, notice that the relative gain of
optimal sampling over periodic sampling consistently grows to
about $67\%$.


\section{Sampling the Ornstein-Uhlenbeck process\label{ornsteinUhlenbeckSection}}

Now we turn to  the case when  the signal is an Ornstein-Uhlenbeck
process:
\begin{equation}
dx_t = a x_t dt + d{W_t}, \ \ t \in [0,T], \label{ou-model}
\end{equation}
with $x_0 = 0$ and $W_t$ being a standard Brownian motion. Again,
the sampling times S = $\{ \tau_1, \ldots , \tau_N  \}$ have to be
an increasing sequence of stopping times with respect to the
$x$-process. They also have to lie within the interval $[0,T]$.
Based on the samples and the sample times, the supervisor
maintains an estimate waveform $\hat{x_t}$ given by
\begin{equation}
\hat{x_t} =
\begin{cases}
0 &  {\text{if}} \ 0 \le t < \tau_1, \\
{x}_{\tau_i} e^{a(t-\tau_i)} &  {\text{if}} \ \tau_i \le t < \tau_{i+1} \le \tau_N, \\
{x}_{\tau_N} e^{a(t-\tau_N)} &  {\text{if}} \ \tau_N \le t \le T.
\end{cases}
\end{equation}
The quality of this estimate is measured by the aggregate squared
error distortion:
\begin{align*}
J^{*}\left( T ,~N\right)& = {\mathbb{E}} \left[ \int_0^T{\left(
x_s - \hat{x}_s\right)}^2 ds
                    \right].
\end{align*}
\subsection{Optimal deterministic sampling}
Just like in the case of Brownian motion, We can show through
mathematical induction that uniform sampling on the interval
$[0,T]$ is the optimal deterministic choice of $N$ samples: For
the induction step, we assume that the optimal choice of $N-1$
deterministic samples over $[T_1,T_2]$ is the uniform one:
\begin{equation*}
\left\{ d_1, d_2, \ldots d_{N}  \right\} = \left\{ T_1 +
i{\frac{T_2 -T_1}{N+1}}
                                         {\Big{\vert}}  i = 1,2, \ldots, N \right\}.
\end{equation*}
The corresponding minimum distortion becomes:
\begin{equation*}
 \left(  N+1 \right) {\frac{e^{2a{\frac{T_2-T_1}{N+1}}} -1}{4a^2}}  -
{\frac{1}{2a}} \left(T_2-T_1\right).
\end{equation*}
\subsection{Optimal Delta sampling}
We do not have an analytical characterization of the performance
of Delta sampling. Let us first address the single sample case.
The performance measure then takes the form
\begin{eqnarray*}
J_{Thresh}\left( T
,~1\right)&=&\Exp\left[\int_0^{\zeta_1}x_t^2+\int_{\zeta_1}^T\left(x_t-\hat{x}_t
\right)^2\,dt\right],\\
&=&\Exp\left[\int_0^{T}x_t^2-2 \int_{\zeta_1}^T x_t\hx_t\,dt +
\int_{\zeta_1}^T(\hx_t)^2\,dt\right].
\end{eqnarray*}
Now notice that the second term can be written as follows
$$
\Exp\left[\int_{\zeta_1}^T
x_t\hx_t\,dt\right]=\Exp\left[\int_{\zeta_1}^T
\Exp[x_t|\cF_{\zeta_1}]\hx_t\,dt\right]=
\Exp\left[\int_{\zeta_1}^T (\hx_t)^2\,dt\right],
$$
where we have used the strong Markov property of $x_t$, and that
for $t>\zeta_1$ we have $\Exp[x_t|\cF_{\zeta_1}]=x_\zeta
e^{-a(t-\zeta_1)}=\hx_t$. Because of this observation the
performance measure takes the form
\begin{align*}
J_{Thresh}\left( T ,~1\right) & =  \Exp\left[ \int_0^T
x_t^2\,dt-\int_{\zeta_1}^T(\hx_t)^2\,dt
\right],\\
&={\frac{e^{2aT}-1-2aT}{4a^2}} - \Exp\left[
x_{\zeta_1}^2\frac{e^{2a(T-\zeta_1)}-1}{2a}
\right],\\
&=T^2\left\{\frac{e^{2aT}-1-2aT}{4(aT)^2} - \Exp\left[
\frac{x_{\zeta_1}^2}{T}\frac{e^{2(aT)(1-\zeta_1/T)}-1}{2(aT)}
\right]\right\},\\
&=T^2\left\{\frac{e^{-2\bar{a}}-1+2\bar{a}}{4\bar{a}^2} -
\Exp\left[
-{\bar{x}_{\bar{\zeta_1}}^2}\frac{e^{2\bar{a}(1-\bar{\zeta_1})}-1}{2\bar{a}}
\right]\right\},
\end{align*}
where,
\begin{equation}
\bar{t}=\frac{t}{T};~~\bar{a}=aT;~~\bar{x}_{\bar{t}}=\frac{x_{\frac{t}{T}}}{\sqrt{T}}.
\label{eq2p}
\end{equation}
We have $\bar{x}$ satisfying the following SDE:
$$
d\bar{x}_{\bar{t}}=-\bar{a}\bar{x}_{\bar{t}}d\bar{t}+dw_{\bar{t}}.
$$
This suggests that, without loss of generality, we can limit
ourselves to the normalized case $T=1$ since the case $T\ne1$ can
be reduced to the normalized one by using the transformations in
(\ref{eq2p}).
In fact, we can solve the single sampling problem on $[0,1]$ to
minimize:
\begin{align}
J_{Thresh}\left( 1 ,~1\right)  &=\left\{\frac{e^{-2a}-1+2a}{4a^2}
- \Exp\left[ -{x_{\zeta_1}^2}\frac{e^{2a(1-\zeta_1)}-1}{2a}
\right]\right\}. \label{reducedExpressionOU}
\end{align}

We carry over the definitions for threshold sampling times from
section \ref{wiener-multiple-level}. We do not have series
expansions like for the case of the Wiener process. Instead we
have a computational procedure that involves solving a PDE initial
and boundary value problem~\cite{kushnerNumericalMethodsBook}. We have a nested sequence of
optimization problems. The choice at each stage being the non-zero
level $\delta_i$. For $N=1$, the distortion corresponding to a
chosen $\delta_1$ is given by:
\begin{align*}
{\frac{e^{2a}-1 -2a}{4a^2}}
 - {\frac{\delta_1^2}{2a}} \Exp \left[ e^{2a(1-\zeta_1)}-1 \right]
 &
= {\frac{e^{2a}-1-2a }{4a^2}}
 - {\frac{\delta_1^2}{2a}}\left\{ e^{2a}\left( 1+2aU^1(0,0)\right)-1\right\},
\end{align*}
where the function $U^1(x,t)$ defined on $[ -\delta_1, \delta_1 ]
\times [0,1]$ satisfies the PDE:
\begin{equation*}
{\frac{1}{2}}U_{xx}^1 + a x U_x^1 + U_t^1 + e^{-2at} = 0,
\end{equation*}
along with the boundary and initial conditions:
\begin{equation*}
\begin{cases}
U^1(-\delta_1,t) = U^1(\delta_1,t) = 0 \ \ & \ \ {\text{for}} \ \ t \in [0,1], \\
U^1(x,1) = 0 \ \ & \ \ {\text{for}} \ \ x \in [ -\delta_1,
\delta_1 ].
\end{cases}
\end{equation*}
We choose the optimal $\delta_1$ by computing the resultant
distortion for increasing values of $\delta_1$ and stopping when
the cost stops decreasing and starts increasing. Note that the
solution $U(0 , t )$ to the PDE also furnishes us with  the performance
of the $\delta_1$-triggered sampling over $[t,1]$. We will use
this to solve the multiple sampling problem.

Let the optimal policy of choosing $N$ levels for sampling over
$[T_1,1]$ be given where $0\le T_1 \le 1$. Let the resulting
distortion be also known as a function of $T_1$. Let this known
distortion over $[T_1,1]$ given $N$ level-triggered samples be
denoted $J^*_{Thresh}\left( 1 - T_1 ,~N\right)$. Then, the $N+1$
sampling problem can be solved as follows. Let $U^{N+1}(x,t)$
satisfy the PDE:
\begin{equation*}
{\frac{1}{2}}U_{xx}^{N+1} + a x U_x^{N+1} + U_t^{N+1}  = 0,
\end{equation*}
along with the boundary and initial conditions:
\begin{equation*}
\begin{cases}
U^{N+1}(-\delta_1,t) = U^{N+1}(\delta_1,t) = J^*_{Thresh}\left( 1
- t ,~N\right) \ \
& \ \ {\text{for}} \ \ t \in [0,1], \\
U^{N+1}(x,1) = 0 \ \ & \ \ {\text{for}} \ \ x \in [ -\delta_1,
\delta_1 ].
\end{cases}
\end{equation*}
Then the distortion is
given by:
\setlength{\multlinegap}{0pt}
\begin{multline*}
{\frac{e^{2a}-1-2a}{4a^2}}
 - {\frac{\delta_1^2}{2a}} \Exp \left[ e^{2a(1-\zeta_1)}-1  +
   {\frac{e^{2a (1-\zeta_1) }-1 }{4a^2}} - {\frac{1-\zeta_1}{2a}} \right]
+  \Exp \left[ J^*_{Thresh}\left( 1 - \zeta_1 ,~N\right) \right]
\\ = {\frac{e^{2a}-1 }{4a^2}} - {\frac{1}{2a}}
 - {\frac{\delta_1^2}{2a}}\left\{ e^{2a}\left(
 1+2aU^1(0,0)\right)-1\right\}
 + U^{N+1}(0,0).
\end{multline*}
We choose the optimal $\delta_1$ by computing the resultant
distortion for increasing values of $\delta_1$ and stopping when
the distortion stops decreasing.

\subsection{Optimal Sampling}
We do not have analytic expressions for the minimum distortion
like in the Brownian motion case. We have a numerical computation
of the minimum distortion by finely discretizing time and solving
the discrete-time optimal stopping problems.

By discretizing time, we get random variables $x_1,\ldots,x_M$,
that satisfy the AR(1) model below. For $1 \le n \le M$,
$$
x_n=e^{a\delta} x_{n-1} + w_n,~~w_n \sim
{\mathcal{N}}\left(0,\frac{e^{2a\delta} -1}{2a}\right);~~1\le n\le
M.
$$
The noise sequence $\{w_n\}$ is i.i.d. and Gaussian.

Sampling exactly once in discrete time means selecting a sample
$x_\nu$ from the set of $M+1$ sequentially available random
variables $x_0,\ldots,x_M$, with the help of a stopping time
$\nu\in\{0,1,\ldots,M\}$. We can define the optimum cost to go
which can be analyzed as follows. For $n=M,M-1,\ldots,0$, using
equation~(\ref{reducedExpressionOU}),
\begin{eqnarray*}
V_n^1(x)&=&\sup_{n\le\nu\le M} \Exp\left[
{x_{\nu}^2}\frac{e^{2{a}\delta(M-\nu)}-1}{2{a}}\Big|x_n=x\right]\nonumber,\\
&=&\max\left\{ {x^2}\frac{e^{2{a}\delta(M-n)}-1}{2{a}}, \
\Exp[V_{n+1}^1(x_{n+1})|x_n=x] \right\}.
\end{eqnarray*}
The above equation provides a (backward) recurrence relation for
the computation of the single sampling value function $V_n^1(x)$.
Notice that for values of $x$ for which the l.h.s.
 exceeds the r.h.s. we stop and sample, otherwise
we continue to the next time instant. We can prove by induction
that the optimum policy is a {\em{time-varying threshold}} one.
Specifically for every time $n$ there exists a threshold
$\lambda_n$ such that if $|x_n|\ge\lambda_n$ we sample, otherwise
we go to the next time instant. The numerical solution of the
recursion presents no special difficulty if $a\le 1$. For $a > 1$,
we need to use careful numerical integration schemes in order to
minimize the computational errors~\cite{kushnerNumericalMethodsBook}. If
$V_n^1(x)$ is sampled in $x$ then this function is represented as
a vector. In the same way we can see that the conditional
expectation is reduced to a simple matrix-vector product. Using
this idea we can compute numerically the evolution of the
threshold $\lambda_t$ with time. The minimum expected distortion
for this single sampling problem is:
$$
{\frac{e^{2aT}-1-2aT}{4a^2}} - V_0^1(0).
$$

For obtaining the solution to the $N+1$-sampling problem, we use
the solution to the $N$-sampling problem. For $n = M, M-1, \ldots
0$,
\begin{align*}
V_n^{N+1}(x)& = \sup_{n\le\nu\le M}
     \Exp\left[ V_{\nu}^N(0) +
             {x_{\nu}^2}\frac{e^{2{a}\delta(M-\nu)}-1}{2{a}}\Big|x_n=x\right]\nonumber,\\
& = \max\left\{  V_n^N(0) +
{x^2}\frac{e^{2{a}\delta(M-n)}-1}{2{a}},
                 V_{n+1}^N(0) + \Exp \left[ V_{n+1}^1(x_{n+1})|x_n=x \right]
          \right\}    .
\end{align*}
\subsection{Comparisons}
    \begin{figure}[t]
    \begin{centering}
    \subfigure[A stable case]%
    {%
    \includegraphics[width=0.45\textwidth]{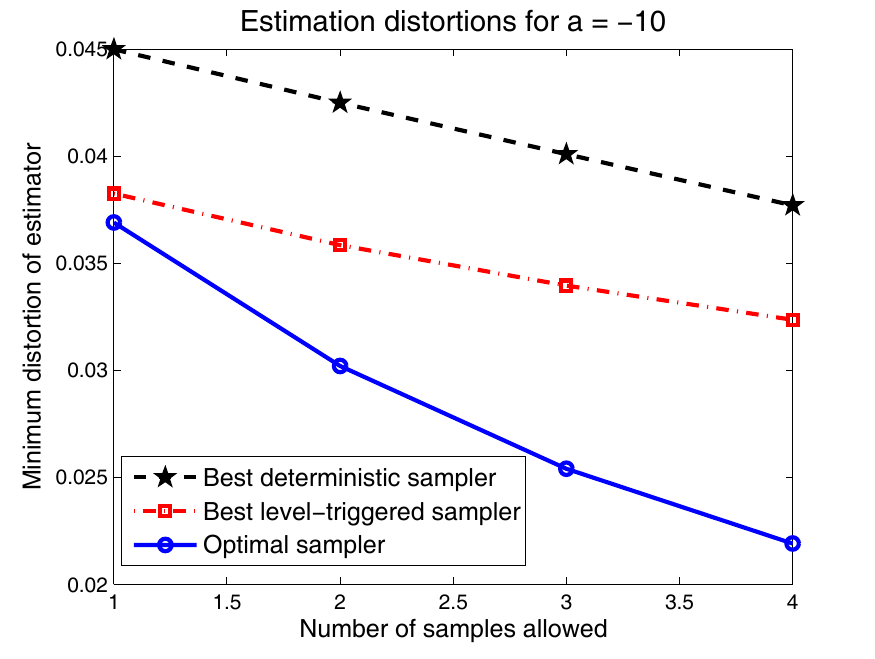}%
    }%
    \subfigure[A stable case]%
    {%
    \includegraphics[width=0.45\textwidth]{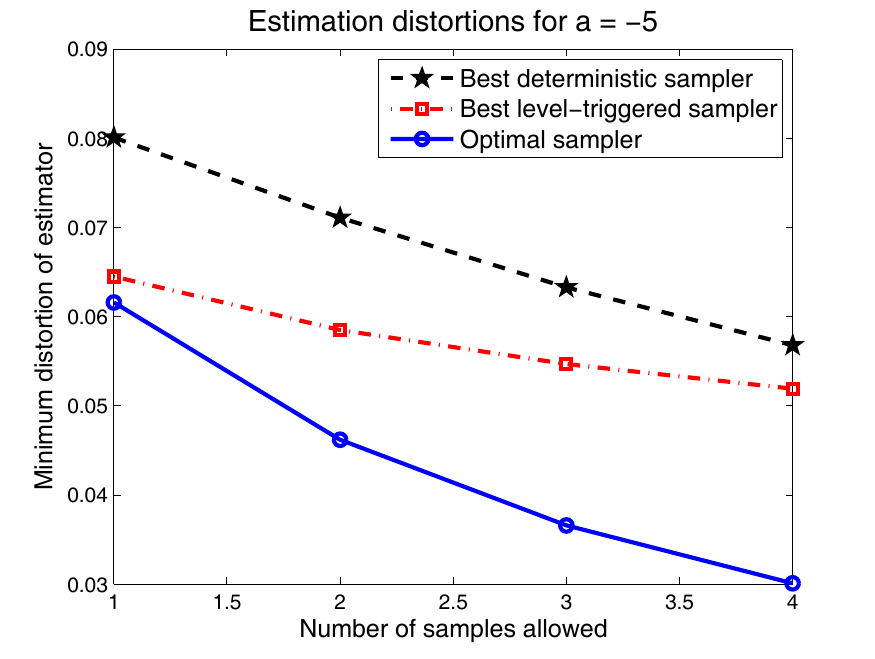}%
    }\\ %
    \subfigure[A stable case]%
    {%
    \includegraphics[width=0.45\textwidth]{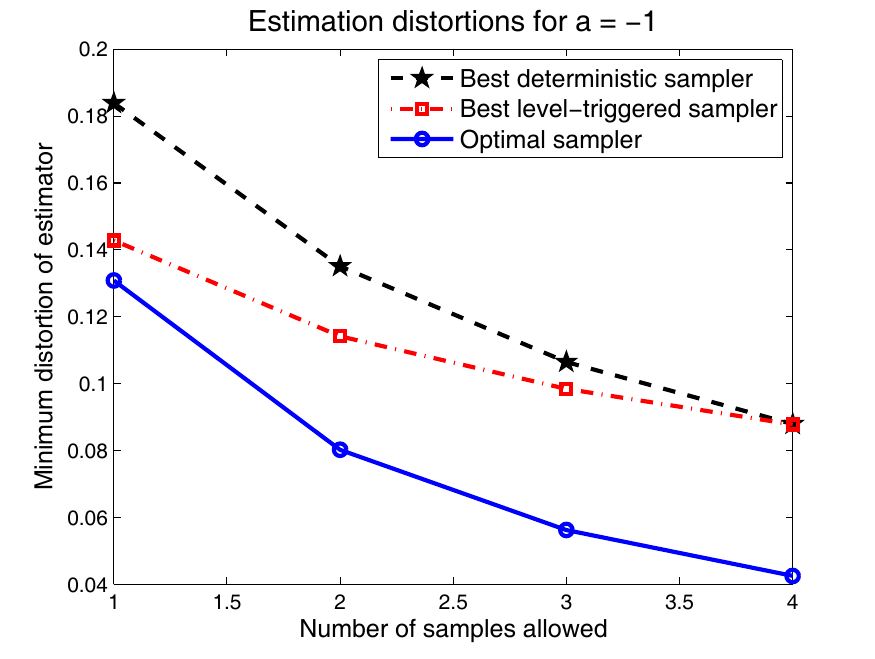}%
    }%
    \subfigure[An unstable case]%
    {%
    \includegraphics[width=0.45\textwidth]{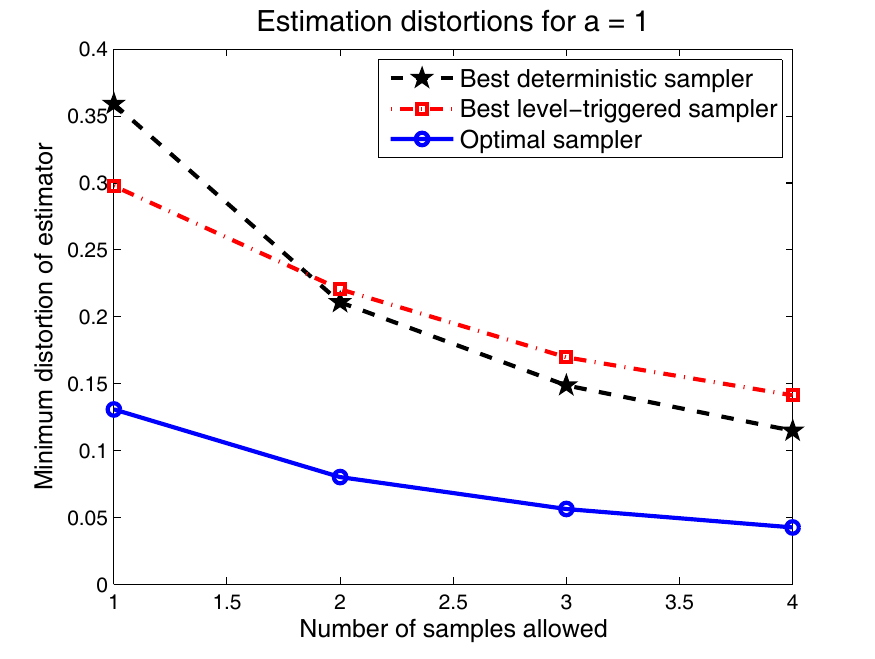}%
    }%
    \caption{The minimum distortions offered by the three sampling methods.
    In the stable regime, Delta sampling is more efficient than deterministic
    sampling when the number of allowed samples is greater than one but still small.
    In the unstable regime, deterministic sampling always beats Delta sampling.}
    \label{ouComparisons}
    \end{centering}
    \end{figure}
Figure~\ref{ouComparisons} shows the result of the numerical
computations for a few stable plants and a single unstable plant.
Again, Delta sampling is not competitive, in the stable
cases, it does provide lower distortion than periodic sampling
when the size of the sample budget is small.

\section{Summary and extensions\label{conclusionsSection}} We
have set up the problem of efficient sampling as an optimal
sequential sampling problem. We conjecture that the
estimator under optimal sampling is the simple least-squares
one under deterministic sampling.
By fixing the estimate to me that under deterministic sampling,
we reduce the optimization into a
tractable stopping time problem.
Our conjecture of course needs to be proved or disproved.

We have furnished methods to obtain good sampling policies for the
finite horizon state estimation problem. When the signal 
is a Brownian motion, we have analytic solutions.
When the signal is an Ornstein-Uhlenbeck process, we have provided
computational recipes to determine the best sampling policies and
their performances. In both cases, Delta sampling performs poorly
with it distortion staying boundedly away from zero even as the
sample budget increases to infinity. This means that the designer
cannot just settle for `natural' event-triggered schemes without
further investigation. In particular, a scheme optimal in the
infinite horizon may perform badly on finite horizons.

The approach adopted in this paper leads us to also consider some
sampling and filtering problems with multiple sensors. These can
possibly be solved in the same way as the single sensor problem.
The case where the samples are not reliably transmitted but can be
lost in transmission is computationally more involved. There, the
relative performances of the three sampling strategies is unknown.
However, in principle, the best policies and their performances
can be computed using nested optimization routines like we have
used in this paper.

Another set of unanswered questions involve the performance of
these sampling policies when the actual objective is not filtering
but control or signal detection based on the samples. It will be
very useful to know the extent to which the overall performance is
decreased by using sampling designs that achieve merely good
filtering performance. The communication constraint we treated in
this paper was a hard limit on the number of allowed samples.
Instead, we could use a soft constraint: a limit on the expected
number of samples. We could also study the effect of mandatory
minimum intervals between successive sampling times. Extension to
nonlinear systems is needed as are extensions to the case of
partial observations at the sensor. One could follow the attack
line sketched at the end of
section~\ref{problemFormulationSection}.


\appendix\label{appendix}
\section*{A. Optimal sampling with non-standard MMSE estimates
\label{appendixPoisson}}
The conjecture expressed by equation~\ref{conjecture} 
 conforms to our intuition about scalar Gaussian diffusions.
Here, we give an example of a
well-behaved and widely used stochastic process for which, the
optimal sampling policy leads to an MMSE estimate which is
different from that under deterministic sampling.
Its increments do not have symmetric PDFs.  For convenience,
we consider an infinite horizon repeated sampling problem where
the communication constraint is a limit on the average sampling
rate.

Choose the state process to be the Poisson counter $N_t$, a
continuous time Markov chain. This is a non-decreasing process
which starts at zero and takes integer values. Its sample paths
are piecewise constant and RCLL. The sequence of times between
successive jumps is IID,  with an exponential
distribution of parameter~$\lambda$.

Under any deterministic sampling rule, the MMSE estimate is
piecewise linear with slope $\lambda$, and has the form:
    \begin{align}
    {\hat{N}}_t & = N_{d_{latest}} + \lambda \left(  t- d_{latest}
    \right), \label{deterMeanPoisson}
    \end{align}
where, $d_{latest}$ is the latest sampling instant as of time $t$.
The optimal sampling policy leads to an MMSE estimate which is
different.

Stipulate that the constraint on the average sampling rate is
equal to  $\lambda$ the parameter of the Poisson process. Consider
the following sampling policy whose MMSE estimate is of the
zero-order hold type:
    \begin{gather}
    {\hat{N}}_t  = N_{{\tau}_{latest}}, \\
    \begin{cases}
        \tau_0 & = 0, \\
        \tau_{i+1} & = \inf \left\{ t \left|  t > \tau_i , N_t > N_{\tau_i} \right.\right\} ,  \ \forall i ~\ge ~0,\\
        \tau_{latest} & = \max \left\{ \tau_i \left|  \tau_i \le t \right. \right\}
    \end{cases}
    \label{optimalSamplingPoisson}
    \end{gather}
This sampling rule with its MMSE estimate ${\hat{N}}_t$ leads to
an error signal which is identically zero. We also have that
    \begin{align*}
    \bbE \left[ \tau_{i+1} - \tau_{i} \right] & = {\frac{1}
    {\lambda}}, \ \forall ~i ~\ge ~0,
    \end{align*}
and so, the communication constraint is met. On the other hand,
the conventional MMSE estimate (\ref{deterMeanPoisson}) would
result in a non-zero average squared error distortion.

Suppose now that the distortion criterion is not the average value
of the squared error but of the lexicographic distance:
    \begin{align*}
    D \left( N_t, {\hat{N}}_t \right)  & =
    \begin{cases}
        0 &  {\text{if}} \ N_t = {\hat{N}}_t , \\
        1 &  {\text{otherwise}},
    \end{cases}
    \end{align*}
then, under deterministic sampling, the maximum likelihood
estimate:
    \begin{align*}
    {\bar{N}}_t & = N_{d_{latest}} + \bigl\lfloor \lambda \left(  t-
    d_{latest} \right) \bigr\rfloor ,
    \end{align*}
minimizes the average lexicographic distortion which will be
non-zero. However, the adaptive policy
(\ref{optimalSamplingPoisson}) provides zero error reconstruction
if the allowed average sampling rate is more than~$\lambda$.


\section*{B. Threshold sampling once
\label{appendixPoisson}}
We drop the subscript $N$ for the terminal sample time:
\begin{align*}
\tau_\delta & = \inf_t\left\{t:\left|x_t - \hat{x}_t
\right|=\delta \right\},
\end{align*}
and its corresponding threshold $\delta$. Here, $\delta$ is a
threshold independent of the data acquired after time 0. Our goal
is to compute the performance measure for any non-negative choice
of the threshold and then select the one that minimizes the
estimation distortion:
    \begin{align}
    J_{Thresh}\left( T ,~1 \right) \left( \delta \right) & =
     \Exp \left[
                   \int_{0}^{ \tau_{\delta}\wedge T } x_s^2 ds
         + \int^{T}_{ \tau_{\delta}\wedge T }
         {\left(x_s-x_{\tau_{\delta}\wedge T}\right)}^2 ds
                    \right].\nonumber\\%
\intertext{By using iterated expectations on the second term, we
           get:}%
    J_{Thresh}\left( T ,~1 \right) \left( \delta \right) & =
     \Exp \left[
                   \int_{0}^{ \tau_{\delta}\wedge T } x_s^2 ds
         +  \Exp \left[ \left. \int^{T}_{ \tau_{\delta}\wedge T }
         {\left(x_s-x_{\tau_{\delta}\wedge T}\right)}^2 ds
                    \right\rvert { \tau_{\delta}\wedge T } , x_{ \tau_{\delta}\wedge T } \right] \right],\nonumber\\
     & =
     \Exp \left[
                   \int_{0}^{ \tau_{\delta}\wedge T } x_s^2 ds
         +  \int^{T}_{ \tau_{\delta}\wedge T } \Exp \left[ \left.
         {\left(x_s-x_{\tau_{\delta}\wedge T}\right)}^2
         \right\rvert { \tau_{\delta}\wedge T } , x_{ \tau_{\delta}\wedge T } \right] ds
                     \right],\nonumber\\
     & =
     \Exp \left[
                   \int_{0}^{ \tau_{\delta}\wedge T } x_s^2 ds
         +  \int^{T}_{ \tau_{\delta}\wedge T }  \left( s - \tau_{\delta}\wedge T \right)ds
                     \right],\nonumber\\
     & =
     \Exp \left[
                   \int_{0}^{ \tau_{\delta}\wedge T } x_s^2 ds
         +  { \frac{1}{2} }  {\left[ { \left(  T - \tau_{\delta}
                                                  \right) }^{+}
           \right]}^{2}
                     \right].
    \label{thresholdRunningPlusTerminal}
    \end{align}
We have thus reduced the distortion measure to a standard form
with a running cost and a terminal cost. We will now take some
further steps and reduce it one with a terminal part alone. Notice
that:
    \begin{gather*}
    d \left[\left(T-t\right) x_t^2 \right] = -x_t^2 dt + 2
    \left(T-t\right) x_t d{x_t} + \left(T-t\right) dt,
    \end{gather*}
which leads to the following representation for the running cost
term:
    \begin{align}
     \Exp \left[   \int_{0}^{ \tau_{\delta}\wedge T } x_s^2 ds \right]
    &  = \Exp \left[  \left(  T - \tau_{\delta}\wedge T  \right)
       x_{\tau_{\delta}\wedge T}^2 + {\frac{T^2}{2}} - {\frac{1}{2}} {\left(
             T - \tau_{\delta}\wedge T  \right)}^2
             \right], \nonumber \\
    &  =  {\frac{T^2}{2}} - \Exp \left[  x_{\tau_{\delta}\wedge T}^2 { \left(  T -
    \tau_{\delta}
                                                  \right) }^{+}  +
       {\frac{1}{2}}  {\left[  { \left(  T - \tau_{\delta}
                                                  \right) }^{+}
           \right]}^{2}
    \right], \label{reduceRunningCost} \\
        &  =  {\frac{T^2}{2}} - \Exp \left[  \delta^2 { \left(  T -
    \tau_{\delta}
                                                  \right) }^{+}  +
       {\frac{1}{2}}  {\left[  { \left(  T - \tau_{\delta}
                                                  \right) }^{+}
           \right]}^{2}
    \right].\nonumber
    \end{align}
Note that equation \ref{reduceRunningCost} is valid even if we
replace $\tau_\delta$ with a random time that is a stopping time
w.r.t. the $x$-process. Thus, the cost
(\ref{thresholdRunningPlusTerminal}) becomes:
    \begin{align}
    J_{Thresh}\left( T ,~1 \right) \left( \delta \right) & =
     {\frac{T^2}{2}} -  \delta^2  \Exp  \left[  {
    \left(  T - \tau_{\delta} \right) }^{+}   \right].
    \label{singleThresholdFinalForm}
    \end{align}
If we can describe the dependence of the expected residual time
${\smash{\Exp \left[{\smash{ { \left(  T - \tau_{\delta} \right)
}^{+}}} \right]}}$  on the threshold $\delta$, then, we can
parametrize the cost purely in terms of $\delta$. Had we known the
PDF of $\tau_\delta$ the computation of the expectation of the
difference $(T-\tau_\delta)^+$ would have been easy. Unfortunately
the PDF of the hitting time $\tau_\delta$ does not have a closed
form solution. There exists a series representation we can find in
page 99 of Karatzas and Shreve~\cite{karatzas-shreve} which is:
    \begin{align*}
    f_{\tau_{\delta}}(t) & = \delta\sqrt{\frac{2}{\pi
    t^3}}\sum_{k=-\infty}^\infty
    (4k+1)e^{-\frac{(4k+1)^2\delta^2}{2t}}.
    \end{align*}
This series is not integrable and so it cannot meet our needs.
Instead we compute the moment generating function of
$(T-\tau_\delta)^+$ and thereby compute the expected distortion.
\section*{B. Statistics of a threshold Hitting time curtailed by a time-out T\label{thresholdCalculationAppendix}}
We start by deriving the Moment generating function of the first
hitting time $\tau_\delta$:
\begin{align*}
\tau_\delta & = \inf_t\left\{t:\left|x_t - \hat{x}_t
\right|=\delta \right\},
\end{align*}
with the initial condition $x_0 - \hat{x}_0 = w_o$.
\begin{lemma} \label{MGFlemma} If $\tau_\delta$ is the first hitting time of
$|x_t - \hat{x}_t|$ at the threshold $\delta$, then,
$$
\Exp[e^{-s\tau_\delta}]=\frac{\cosh(w_0\sqrt{2s})}{\cosh(\delta\sqrt{2s})}=F_\tau(s).
$$
\end{lemma}
\begin{proof}
Consider the ${\mathbb{C}}^2$ function
$h(w,t)=e^{-st}[1-\cosh(\sqrt{2s}w)/\cosh(\sqrt{2s}\delta)]$ and
apply It\^o calculus on $h\left(w_t,t\right)$. We can then
conclude that
\begin{align*}
\Exp\left[h\left(w_{\tau_\delta},{\tau_\delta}\right)\right]-h(w_0,0)
& =
\Exp\left[\int_0^{\tau_\delta}[h_t(w_t,t)+0.5h_{ww}(w_t,t)]\,dt\right]=\Exp[e^{-s\tau_\delta}]-1,
\end{align*}
from which we immediately obtain the desired relation because of
the boundary condition: $h(w_{\tau_\delta},{\tau_\delta})=0$.
\end{proof}
Lemma\,\ref{MGFlemma} suggests that the PDF of the random variable
$\tau_\delta$ can be computed as $f_\tau(t)=\mL^{-1}(F_\tau(s))$,
that is, the inverse Laplace transform of $F_\tau(s)$. Invoking
the initial condition $w_0=0$, we can then write
\begin{align*}
\Exp[(T-\tau_\delta)^+]& =\int_0^T(T-t)f_\tau(t)dt =
\int_0^T(T-t)\left[\frac{1}{2\pi j}\oint F_\tau(s)e^{st}\,ds\right]dt\\
& = \frac{1}{2\pi j}\oint F_\tau(s)\left[\int_0^T(T-t)e^{st}\,dt\right]\,ds\\
& = \frac{1}{2\pi j}\oint\frac{e^{sT}-1-sT}{s^2
\cosh(\delta\sqrt{2s})}\,ds.
\end{align*}
The previous integral is on a path that includes the whole left
half of the complex plane.

In order to compute this line integral over the complex plane, we
need to find the poles of the integrand and then apply the residue
theorem. Notice first that $s=0$ is not a pole since the numerator
has a double zero at zero. The only poles come from the zeros of
the function $\cosh(\delta\sqrt{2s})$. Since $\cosh(x)=\cos(jx)$
we conclude that the zeros of $\cosh(\delta\sqrt{2s})$ which are
also the poles of the integrand are:
\begin{gather*}
s_k  = -(2k+1)^2\frac{\pi^2}{8\delta^2},~k=0,1,2,\ldots
\end{gather*}
and they all belong to the negative half plane. This of course
implies that they all contribute to the integral. We can now apply
the residue theorem to conclude that
$$
\Exp[(T-\tau)^+]= \frac{1}{2\pi j}\oint\frac{e^{sT}-1-sT}{s^2
\cosh(\delta\sqrt{2s})}\,ds= \sum_{k\ge 0}
\frac{e^{s_kT}-1-s_kT}{s_k^2}\lim_{s\to
s_k}\frac{s-s_k}{\cosh(\delta\sqrt{2s})}.
$$
In order to find the last limit we can assume that
$s=s_k(1+\epsilon)$ and let $\epsilon\to0$. Then, we can show that
\begin{align*}
\lim_{s\to s_k}\frac{s-s_k}{\cosh(\delta\sqrt{2s})} &
=(-1)^{(k+1)}\frac{4s_k}{(2k+1)\pi}.
\end{align*}
Using this expression, the performance measure of the stopping
time~$\tau_\delta$ takes the following form:
\begin{align*}
J_{Thresh}\left( T ,~1 \right) &={\frac{T^2}{2}} -  \delta^2 \Exp
{\left[  { \left(  T - \tau_{\delta} \right) }^{+}
\right]},\\
&=\frac{T^2}{2}\left\{1-\frac{8\delta^2}{\pi T}\sum_{k\ge
0}(-1)^{(k+1)}
\frac{1}{2k+1}\frac{e^{s_kT}-1-s_kT}{s_k T}\right\}\\
& = \frac{T^2}{2}\phi(\lambda), 
\end{align*}
where, with the change of variables:
$\lambda = \frac{T\pi^2}{8\delta^2}$, we have:
\begin{align*}
 \phi(\lambda) &{\stackrel{\Delta}{=}} 1-\frac{\pi}{\lambda^2}\sum_{k\ge 0}(-1)^{k}
\frac{e^{-(2k+1)^2\lambda}-1+(2k+1)^2\lambda}{(2k+1)^3}, \\
 & = 1-\frac{\pi}{\lambda^2}\sum_{k\ge 0}{
             \frac{ {(-1)^{k}} e^{-(2k+1)^2\lambda} } {(2k+1)^3}}
      +\frac{\pi}{\lambda^2}\sum_{k\ge 0}{
              \frac{ {(-1)^{k}} } {(2k+1)^3}}
      -\frac{\pi}{\lambda}\sum_{k\ge 0}{
             \frac{ {(-1)^{k}}  } {2k+1}}.
\end{align*}%
The final two series in the last equation can be summed
explicitly. To do so, we adopt a summation technique described in
the book of Aigner and
Ziegler~\cite{proofsFromTheBookSecondEdition}. Consider:
\begin{gather*}
\int_{0}^{1}{{\frac{dx}{1+x^2}}} \ = \ \int_{0}^{1}
\left(\sum_{k\ge 0}{
 {(-1)^{k}} {x^{2k}}} \right)dx \ = \
\sum_{k\ge 0}{
 {(-1)^{k}} \int_{0}^{1} {x^{2k}}}dx \ = \ \sum_{k\ge 0}{
             \frac{ {(-1)^{k}}  } {2k+1}}.
\end{gather*}
By an easy evaluation of the definite integral we started with, we
get a sum of ${\tfrac{\pi}{4}}$ for the series $\sum_{k\ge
0}{\tfrac{ {(-1)^{k}} } {2k+1}}$; this result is useful because
the series converges slowly. Proceeding along similar
lines~\cite{beukersKolk} and working with the multiple integral:
\begin{gather*}
\idotsint\displaylimits_{A}{\frac{dx_1\cdots dx_n}{1+{\left(
x_1x_2\cdots x_n\right)}^2}},
\end{gather*}
over the unit hypercube $A = {\left[0,1\right]}^{n}$ in
${\mathbb{R}}^n$, we get an explicit expression for the sum
$\sum_{k\ge 0}{\tfrac{ {(-1)^{k}} } {(2k+1)^n}}$ whenever $n$ is
an odd number. In particular,
\begin{align*}
\sum_{k\ge 0}{\frac{ {(-1)^{k}}  } {(2k+1)^3}} & =
{\frac{\pi^3}{32}}; & \sum_{k\ge 0}{\frac{ {(-1)^{k}}} {(2k+1)^5}}
& ={\frac{5\pi^5}{1536}}.
\end{align*}
This reduces the distortion to:
\begin{gather}
{\boxed{%
 J_{Thresh}\left( T ,~1 \right)
    \ = \ \frac{T^2}{2}\phi(\lambda) \ = \  %
        {\frac{T^2}{2}} \left\{ 1 +{\frac{\pi^4}{32\lambda^2}}%
          -{\frac{\pi^2}{4\lambda}}%
          -{\frac{\pi}{\lambda^2}}
              \sum_{k\ge 0}{ \frac{ {(-1)^{k}}
                  e^{-(2k+1)^2\lambda} }{(2k+1)^3}} \right\},%
}}%
\label{thresholdJforOneSample}
\end{gather}
where, $\lambda = \frac{T\pi^2}{8\delta^2}$.

The estimation distortion due to using $N+1$ samples when $N$ is
non-negative is given through the recursion:
\begin{align*}
J_{Thresh}\left( T ,~N+1 \right)
   & =  {\frac{T^2}{2}} -  \delta^2 \Exp {\left[  { \left(  T - \tau_{\delta_{N+1}}
                                              \right) }^{+}   \right]}
\\
& \ \ \
      - \left( {\frac{1}{2}} - J_{Thresh}\left( T ,~N \right)
             \right) \Exp  \left[ {\left[  { \left(  T - \tau_{\delta_{N+1}}
                         \right) }^{+}   \right]}^{2} \right].
\end{align*}
This recursion decomposes the distortion into an expression
involving the performance of the policies for the last $N$ times
of the first of $N+1$ sampling times over the horizon
$\left[0,T\right]$ , and, the influence of the first of the
sampling times namely, $\tau_{\delta_{N+1}}$. Hence, for our
recursive solution to the multiple stopping problem, we need to
minimize costs like:
\begin{align*}
\Upsilon \left( T , \alpha, \delta \right)
   & {\stackrel{\Delta}{=}} {\frac{T^2}{2}} -  \delta^2 \Exp {\left[  { \left(  T - \tau_{\delta}
                                              \right) }^{+}   \right]}
      - \left( {\frac{1}{2}} - \alpha
             \right) \Exp  \left[ {\left[  { \left(  T - \tau_{\delta}
                         \right) }^{+}   \right]}^{2} \right],
\end{align*}
where $\alpha$ is positive but no greater than 0.5. This requires
an evaluation of the second moment: $\Exp  \left[ {\left[  {
\left( T - \tau_{\delta} \right) }^{+}   \right]}^{2} \right]$.
The evaluation can be done similarly to the one for the first
moment:
\begin{align*}
\Exp\left[{\left[\left(T-\tau\right)^+\right]}^2\right] &
 = {\frac{1}{\pi j}}
 \oint  {\frac{e^{sT}-1-sT-{\frac{1}{2}}s^2T^2}{s^3
\cosh\left(\delta\sqrt{2s}\right)}} \,ds.
\end{align*}
This gives the following expression for the cost $\Upsilon \left(
T , \alpha, \delta \right)$:
\begin{align*}
\Upsilon \left( T , \alpha, \delta \right)
 &=  {\frac{T^2}{2}} \left\{ \phi(\lambda) +  \left[ {\frac{1}{2}} - \alpha
             \right] \psi(\lambda) \right\}, 
\end{align*}
where, $\lambda = \frac{T\pi^2}{8\delta^2}$, and we define
functions $\phi, \psi$ with $\phi$ being the same as it was
earlier in this part of the appendix:
\begin{align*}
 \phi\left(\lambda\right) &{\stackrel{\Delta}{=}} 1-\frac{\pi}{\lambda^2}\sum_{k\ge 0}(-1)^{k}
\frac{e^{-(2k+1)^2\lambda}-1+(2k+1)^2\lambda}{(2k+1)^3}, \\
 & = 1-\frac{\pi}{\lambda^2}\sum_{k\ge 0}{
             \frac{ {(-1)^{k}} e^{-(2k+1)^2\lambda} } {(2k+1)^3}}
      +\frac{\pi}{\lambda^2}\sum_{k\ge 0}{
              \frac{ {(-1)^{k}} } {(2k+1)^3}}
      -\frac{\pi}{\lambda}\sum_{k\ge 0}{
             \frac{ {(-1)^{k}}  } {2k+1}},\\
\intertext{and,}
 \psi\left(\lambda\right) &{\stackrel{\Delta}{=}}
           {\frac{16}{\pi\lambda^2}}\sum_{k\ge 0}(-1)^{k}
       \frac{e^{-(2k+1)^2\lambda}-1+{\left(2k+1\right)}^2
      \lambda-0.5{\left(2k+1\right)}^4\lambda^2}{{\left(2k+1\right)}^5}, \\
& = {\frac{16}{\pi\lambda^2}}\sum_{k\ge 0}{
             \frac{ {(-1)^{k}}  e^{-(2k+1)^2\lambda} } { {\left(2k+1\right)}^5}}
      -{\frac{16}{\pi\lambda^2}}\sum_{k\ge 0}{
              \frac{ {(-1)^{k}} } {{\left(2k+1\right)}^5}}
      +{\frac{16}{\pi\lambda}}\sum_{k\ge 0}{
             \frac{ {(-1)^{k}}  } {{\left(2k+1\right)}^3}}
      -{\frac{8}{\pi}}\sum_{k\ge 0}{
             \frac{ {(-1)^{k}}  } {2k+1}}. 
\end{align*}%

After replacing the summable series with their sums, the
distortion due to multiple samples based on thresholds reduces to
the boxed expression below.  With $\lambda =
\frac{T\pi^2}{8\delta^2}$,
\begin{gather}
{\boxed{%
\begin{split}
%
 J_{Thresh}\left( T ,~N+1 \right)
     & =   %
       {\frac{T^2}{2}} \left\{ 1 +{\frac{\pi^4}{32\lambda^2}}%
          -{\frac{\pi^2}{4\lambda}}%
          -{\frac{\pi}{\lambda^2}}
              \sum_{k\ge 0}{ \frac{ {(-1)^{k}}
                  e^{-(2k+1)^2\lambda} }{(2k+1)^3}} \right.%
\\
& \left. \quad \quad \quad  \ + \bigl( 0.5 -  J_{Thresh}\left( T ,~N \right) \bigr)%
 \left[ {\frac{-5\pi^4}{96\lambda^2}}%
          -{\frac{\pi^2}{2\lambda}}%
          -2  +{\frac{16}{\pi\lambda^2}} \sum_{k\ge 0}{ \frac{ {(-1)^{k}}
                  e^{-(2k+1)^2\lambda} }{(2k+1)^5}}
\right] \right\}.
\end{split}
}}%
\label{thresholdJforManySamples}
\end{gather}

To characterize the statistics of sample budget utilization by
multiple Delta sampling we need to find the probabilities of
threshold crossings before the time-out $T$. Given the limit of
the allowed number of samples N, let $\Xi_N$ be the random number
of samples generated under the multiple delta sampling scheme,
Then, we have:
\begin{align*}
{\mathbb{E}}{\left[ \Xi_N \right]}
   & \ = \ 0 \cdot {\mathbb{P}}{\left[ \tau_{\delta_{{}_{N}}} \ge T \right]}
+\left( 1 + {\mathbb{E}}{\left[ \Xi_{N-1} \right]} \right)\cdot
{\mathbb{P}}{\left[ \tau_{\delta_{{}_{N}}} < T \right]},\\
 & \ = \
\left( 1 + {\mathbb{E}}{\left[ \Xi_{N-1} \right]} \right)\cdot
\left( 1 - {\mathbb{P}}{\left[ \tau_{\delta_{{}_{N}}} \le T
\right]} \right).
\end{align*}
As before, we use the Moment generating function of the hitting
time to obtain:
\begin{gather*}
\Exp\left[\Xi_1 \right] \ =  \ \Exp\left[{\mathbf{1}}_{\left\{
\tau_{\delta_{{}_{1}}}
> T \right\}}
  \right] \ = \
 {\frac{1}{\pi j}}
 \oint  {\frac{e^{sT}-1}{s \cdot
\cosh\left(\delta\sqrt{2s}\right)}} \,ds.
\end{gather*}
With the notation $\lambda = \frac{T\pi^2}{8\delta^2_{{}_{1}}}$,
and evaluating this complex line integral as in previous cases, we
obtain:
\begin{gather}
{\boxed{%
%
 \Exp\left[\Xi_1 \right] \ =  \ {\mathbb{P}}\left[ \tau_{\delta_{{}_{1}}} \le T \right]
      \ = \  %
       1 -{\frac{4}{\pi}}
              \sum_{k\ge 0}{(-1)^{k}{ \frac{
                  e^{-(2k+1)^2\lambda}} {2k+1}}}%
.
}}%
\label{firingProbabilityExpression}
\end{gather}


\section*{Acknowledgments}

This work was supported by the United States Army Research Office
through the Center for Networked Communication and Control Systems
at Boston University Grant No. DAAD 19021 0319, and under the
CIP-URI Grant No. DAAD 190110494. The first author also received
partial financial support from the Swedish Research Council, and
the European Commission.


\bibliographystyle{siam}
\bibliography{mabenRefs}

\end{document}